\documentclass[12pt, a4paper, oneside]{book}
\usepackage[utf8]{inputenc}
\usepackage[english]{babel}
\usepackage{amssymb}
\usepackage{amsmath}
\usepackage{amsthm}
\usepackage{graphicx}
\usepackage{tikz-cd}

\usepackage[document]{ragged2e}

\usepackage{mathptmx}

\usepackage[left=3.5cm, right=2.5cm, top=2.5cm, bottom=2.5cm]{geometry}

\usepackage{tikz}
\usetikzlibrary{matrix, arrows}
\tikzset{node distance=2cm, auto}

\usepackage{fancyhdr}
\fancypagestyle{plain}{
	\fancyhf{}
	\fancyhead[C]{\thepage}
	\pagestyle{fancy}
}

\addto\captionsenglish{}

\DeclareMathOperator{\Ima}{Im}

\usepackage{tocloft}

\usepackage{tabu}

\newtheorem{theorem}{Theorem}[section]
\newtheorem{proposition}[theorem]{Proposition}
\newtheorem{lemma}[theorem]{Lemma}
\newtheorem{corollary}[theorem]{Corollary}
\newtheorem{definition}[theorem]{Definition}
\newtheorem{remark}[theorem]{Remark}
\newtheorem{question}[theorem]{Question}

\emergencystretch=\maxdimen
\frontmatter

\title{ON D-SPACES AND COVERING PROPERTIES}
\author{\textbf{Talal Alrawajfeh} \and \textbf{Hasan Z. Hdeib}}
\date{July 2019}

\begin{document}
\maketitle

\addcontentsline{toc}{chapter}{Contents}
\begingroup
\tableofcontents
\endgroup

\clearpage

\addcontentsline{toc}{chapter}{Abstract}
\thispagestyle{plain}
\begin{center}
	\vspace*{3cm}
    {\Large
    \textbf{On D-spaces and Covering Properties}
	}

	\vspace{0.5cm}
	{\normalsize
	By \\
	\textbf{Talal Alrawajfeh}
	}

	\vspace{0.5cm}
	{\normalsize
	Supervisor \\
	\textbf{Hasan Z. Hdeib}
	}

	\vspace{1.5cm}
	{\Large
	\textbf{Abstract} \\
	}
	\vspace{1cm}
	\begin{justify}
	In this thesis, we introduce the subject of D-spaces and some of its most important open problems which are related to well known covering properties.
	We then introduce a new approach for studying D-spaces and covering properties in general.
	We start by defining a topology on the family of all principal ultrafilters of a set $X$ called the principal ultrafilter topology.
	We show that each open neighborhood assignment could be transformed uniquely to a special continuous map using the principal ultrafilter topology.
	We study some structures related to this special continuous map in the category \textbf{Top}, then we obtain a characterization of D-spaces via this map.
	Finally, we prove some results on Lindel\"of, paracompact, and metacompact spaces that are related to the property D\@.
	\end{justify}
\end{center}
\let\cleardoublepage

\clearpage
\mainmatter

\chapter{Introduction}\label{ch:introduction}
\section{History and Background}\label{sec:history-and-background}

\justify
All notation, concepts, and results of any branch of mathematics that are not defined or explicitly stated in this thesis could be found in the following references: General Topology in Engelking (1989)$\,^{\text{\cite{engelking}}}$, Set theory in Jech (2002)$\,^{\text{\cite{jech}}}$, and Category theory in both Mac Lane (1971)$\,^{\text{\cite{maclane}}}$ and Awodey (2010)$\,^{\text{\cite{awodey}}}$.

A similar notion to an open cover of a space is that of an \textit{open neighborhood assignment}.
If $(X, \tau)$ is a topological space, then a function (or an assignment) $N : X \rightarrow \tau$ is an open neighborhood assignment if $N(x)$ contains $x$ for each $x$ in $X$.
In other words, an open neighborhood assignment takes each point $x$ to an open neighborhood of it in $\tau$.
From this definition, it is clear that $N(X) = \{N(x) : x \in X \}$ is an open cover of $X$.

\textit{D-spaces} were first introduced in Van Douwen and Pfeffer (1979)$\,^{\text{\cite{van-douwen-pfeffer}}}$, where some properties of the Sorgenfrey line were discussed.
The Sorgenfrey line is the topology on $\mathbb{R}$ generated from the base $\mathcal{B} = \{ \left[ a, b \right) : a,b \in \mathbb{R} \}$.
One of the main results in the paper is that every finite power of the Sorgenfrey line is a D-space.
A D-space or the \textit{property D} is defined as follows.

\begin{definition}[D-space]
	A topological space $(X, \tau)$ is a D-space (or has the property D) if for every open neighborhood assignment $N : X \rightarrow \tau$, there is a closed discrete set $D$ satisfying $\bigcup N ( D ) = \bigcup\limits_{d \in D} N ( d ) = X$.
\end{definition}


\begin{theorem}\label{finite-pow-sorgenfrey-is-d}
	Every finite power of the Sorgenfrey line is a D-space.
\end{theorem}

To prove theorem~\ref{finite-pow-sorgenfrey-is-d}, the authors first introduced the notion of a \textit{GLS-space} which is a generalization of the notion of a \textit{left-separated} space.

\begin{definition}[Left-separated space]
	A topological space $(X, \tau)$ is left-separated if there is a well-order $\preceq$ on $X$ such that $\{ y \in X: x \preceq y\}$ is open in $X$ for each $x$ in $X$.
\end{definition}

We will denote the set $\{ y \in X: x \preceq y\}$ as $\Uparrow x$.
Moreover, for any reflexive relation $\preceq$ on a set $E$ and $F \subseteq E$, a point $m \in F$ that satisfies $x \preceq m \Longrightarrow x = m$ for all $x \in F$ is called a $\preceq$-minimal element of $F$.
Hence, this captures the intuitive concept of the smallest element of a subset of $E$ with respect to the relation $\preceq$.

\begin{definition}[GLS-space]
	A topological space $(X, \tau)$ is called a generalized left-separated space (abbreviated GLS-space) if there is a reflexive binary relation $\preceq$, called a GLS-relation, such that
	\begin{enumerate}
		\item every non-empty closed subset has a $\preceq$-minimal element
		\item $\Uparrow x$ is open for each $x \in X$
	\end{enumerate}
\end{definition}

The authors then proved the following theorem.

\begin{theorem}\label{gls-is-d}
	Every GLS-space is a D-space.
\end{theorem}

The proof involves constructing, for each open neighborhood assignment $N:X\rightarrow \tau$, a closed discrete set $D$ that satisfies $\bigcup N(D) = X$ via transfinite recursion.
Theorem~\ref{gls-is-d} has a central role in proving theorem~\ref{finite-pow-sorgenfrey-is-d} because it turns out that every finite power of the Sorgenfrey line is, in fact, a GLS-space as stated in the following theorem.

\begin{theorem}\label{sorgenfrey-is-gls}
	Every finite power of the Sorgenfrey line is a GLS-space.
\end{theorem}

To go on further, we will need to define the following notions first.

\begin{definition}
	A collection of subsets of a topological space $(X, \tau)$ is locally finite if each $x$ in $X$ has an open neighborhood that intersects at most finitely many sets in that collection.
\end{definition}

\begin{definition}
	A collection of subsets of a topological space $(X, \tau)$ is point finite if each $x$ in $X$ is contained in at most finitely many sets in that collection.
\end{definition}

\begin{definition}
	A topological space $(X, \tau)$ is Lindel\"of if every open cover reduces to a countable subcover.
	We do not require the space $X$ to be regular opposed to Engelking (1989)$\,^{\text{\cite{engelking}}}$.
\end{definition}

\begin{definition}
	A topological space $(X, \tau)$ is paracompact if every open cover has a locally finite open refinement.
	We do not require the space $X$ to be Hausdorff.
\end{definition}

\begin{definition}
	A topological $(X, \tau)$ is metacompact if every open cover has a point finite open refinement.
	We do not require the space $X$ to be Hausdorff.
\end{definition}

The authors of the paper also tried but failed to find a space with a covering property as least as strong as metacompactness or subparacompactness which is not a D-space.
The existence of a regular space that is Lindel\"of, paracompact, metacompact, meta-Lindel\"of, or subparacompact and which is not a D-space still remains unknown to this day even if the covering property was introduced hereditarily (see Gruenhage, 2011)$\,^{\text{\cite{gruenhage}}}$.

Some results were obtained on D-spaces after Van Douwen and Pfeffer in Borges and Wehrly (1991)$\,^{\text{\cite{borges-wehrly}}}$.
They proved the following theorems on D-spaces.

\begin{theorem}\label{semi-start-is-d}
	Semi-stratifiable spaces are D-spaces.
\end{theorem}

\begin{theorem}
	Perfect inverse images of D-spaces are D-spaces.
\end{theorem}

\begin{theorem}
	Closed images of D-spaces are D-spaces.
\end{theorem}

\begin{theorem}\label{mon-normal-d-are-paracompact}
	Monotonically normal D-spaces are paracompact.
\end{theorem}

Literature on \textit{stratifiable spaces} and \textit{semi-stratifiable spaces} could be found in Borges (1966)$\,^{\text{\cite{borges}}}$ and Henry (1971)$\,^{\text{\cite{henry}}}$.
Basically, A semi-stratifiable space is defined as follows.


\begin{definition}
	A topological space $(X, \tau)$ is a semi-stratifiable space if, to each open set $U$ of $X$, one can assign a sequence $\{ F_n \}_{n=1}^{\infty}$ of closed sets of $X$ such that

	\begin{enumerate}
		\item $\bigcup\limits_{n=1}^{\infty} F_n = U$
		\item $F_n \subseteq G_n$ for all $n \in \mathbb{N}$ whenever $U \subseteq V$, where $V$ is an open set of $X$ and $\{ G_n \}_{n=1}^{\infty}$ is the sequence of closed sets associated with $V$ such that $\bigcup\limits_{n=1}^{\infty} G_n = V$.
	\end{enumerate}
\end{definition}

The proof of theorem~\ref{semi-start-is-d} also involves applying transfinite recursion to construct a closed discrete set $D$ for every open neighborhood assignment.

The notion of \textit{monotonically normal spaces}, as mentioned in theorem~\ref{mon-normal-d-are-paracompact}, was first introduced in Borges (1966)$\,^{\text{\cite{borges}}}$ without a name, then it was named by P. L. Zenor (Heath, et al., 1973)$\,^{\text{\cite{heath-lutzer-zenor}}}$.

\begin{definition}
	A space $(X, \tau)$ is a monotonically normal space if there is a function $G$ which assigns to each ordered pair $(H, K)$ of disjoint closed sets of $X$ an open set $G(H, K)$ such that
	\begin{enumerate}
		\item $H \subseteq G(H, K) \subseteq \overline{G(H, K)} \subseteq X \setminus K$
		\item if $(H^{*}, K^{*})$ is any other pair of disjoint closed sets having $H \subseteq H^{*}$ and $K \supseteq K^{*}$, then $G(H, K) \subseteq G(H^{*}, K^{*})$
	\end{enumerate}
	The function $G$ is called a monotone normality operator for $X$.
\end{definition}

The same authors thought they have proved that every $T_1$ Lindel\"of space $X$ is a D-space in Borges and Wehrly (1996)$\,^{\text{\cite{borges-wehrly-qa}}}$.
However, they had made an assumption that wasn't justified in the proof, so they published another paper afterwards announcing they've not proved that statement (Borges and Wehrly, 1998)$\,^{\text{\cite{borges-wehrly-correction}}}$.
This will be explained in more detail in the following section.

For any totally ordered set $(S, \preceq)$, we can define a topology on $S$ called the \textit{order topology} which is generated by the subbase of open rays of $S$, i.e.\ all the sets ${ a^{\rightarrow} = \{ x : a \preceq x \} }$ and $b^{\leftarrow} = \{ x : x \preceq b \}$.
The space $(S, \preceq, \tau)$ where $\tau$ is the order topology on $S$ is called a \textit{linearly ordered space}.
A \textit{generalized ordered space} $(S, \preceq, \tau)$ on $S$ where $\preceq$ is a total order and $\tau$ is a topology on $S$, is a space that could be embedded into a linearly ordered space $(S', \preceq', \tau')$ where the embedding is both an order embedding and a topological embedding.
Another great result on D-spaces could be found in Van Douwen and Lutzer (1997)$\,^{\text{\cite{van-douwen-lutzer}}}$ which states the following.

\begin{theorem}\label{go-is-d}
	A generalized ordered space $X$ is paracompact iff $X$ is a D-space.
\end{theorem}

As an example, since the real line with the usual topology $(\mathbb{R}, \leq, \tau_u)$ is a linearly ordered space, and thus is a generalized ordered space;
and since it is also a paracompact space, then by theorem~\ref{go-is-d}, $(\mathbb{R}, \leq, \tau_u)$ is a D-space.

An approach was proposed in Fleissner and Stanley (2001)$\,^{\text{\cite{fleissner-stanley}}}$ to make some proofs concerning D-spaces easier by introducing the notion of \textit{U-sticky sets}.

\begin{definition}
	Let $U: X \rightarrow \tau$ be an open neighborhood assignment and $D \subseteq X$.
	The set $D$ is said to be $U$-sticky if $D$ is closed discrete and satisfies \[\forall x \in X\left[U(x) \cap D \neq \phi \Longrightarrow x \in \bigcup U(D)\right]\]
\end{definition}

They also defined an order relation $\preceq_U$ on U-sticky sets of $X$ as follows.

\begin{definition}
	Let $U: X \rightarrow \tau$ be an open neighborhood assignment and \linebreak ${ \mathcal{D}(U) = \{ D \subseteq X : D \text{ is U-sticky} \} }$.
	For any $D, D' \in \mathcal{D}(U)$, if
	\begin{enumerate}
		\item $D \subseteq D'$
		\item $(D' \setminus D) \cap \bigcup U(D) = \phi$
	\end{enumerate}
	, then $D \preceq_{U} D'$.
	Moreover, $\mathcal{C}$ is a $\preceq_{U}$-chain if $\mathcal{C} \subseteq \mathcal{D}(U)$ and for every $D, D' \in \mathcal{D}(U)$, either $D \preceq_{U} D'$ or $D' \preceq_{U} D$.
\end{definition}

An interesting fact about $\preceq_{U}$-chains is that for any $\preceq_{U}$-chain $\mathcal{C}$, we have $\bigcup \mathcal{C} \in \mathcal{D}(U)$, i.e.\ $\bigcup \mathcal{C}$ is a U-sticky set.
This is true since $(\mathcal{D}(U), \preceq_U)$ satisfies the hypothesis of Zorn's lemma.
Using these definitions, they found an equivalent statement to the definition of D-spaces.

\begin{theorem}\label{u-stickey-d-space}
	A space $(X, \tau)$ is a D-space iff for every open neighborhood assignment $U:X \rightarrow \tau$, for every $D \in \mathcal{D}(U)$, and for every $x \in X$, there exists $D' \in \mathcal{D}(U)$ such that $x \in \bigcup U(D')$ and $D \preceq_U D'$.
\end{theorem}

Applying theorem~\ref{u-stickey-d-space}, the authors could provide shorter proofs for results that were already established by others.

In Arhangel'skii and Buzyakova (2002)$\,^{\text{\cite{arhangelskii-buzyakova}}}$, it was proved that if a regular topological space $(X, \tau)$ is a union of finitely many metrizable subspaces, then $X$ is a D-space.
This implies that every metrizable space is a D-space (and thus, all metric spaces are D-spaces).
To prove this result, the authors had to first prove a central theorem in their paper that depends on the concept of a \textit{point countable base}.

\begin{definition}
	Let $(X, \tau)$ be a topological space.
	A base $\mathcal{B}$ of $X$ is a point countable base if each point of $X$ is contained in at most countably many open sets of $\mathcal{B}$.
\end{definition}

The following is the statement of their central theorem.

\begin{theorem}
	Every space $X$ with a point countable base $\mathcal{B}$ is a D-space.
\end{theorem}

Buzyakova independently proved in Buzyakova (2002)$\,^{\text{\cite{buzyakova-1}}}$ a stronger result than ``all metrizable spaces are D-spaces'' as stated in the following theorem.

\begin{theorem}\label{strong-simga-is-d}
	Every strong $\Sigma$ space is D-space.
\end{theorem}

A family of subsets of a topological space $(X, \tau)$ is called a $\sigma$-locally finite family if it can be represented as a countable union of locally finite families of subsets of $X$.

\begin{definition}
	A space $X$ is said to be a strong $\Sigma$ space if there exists a $\sigma$-locally finite family $\mathcal{F}$ of closed sets in $X$ and a cover $\mathcal{K}$ of compact subsets of $X$, such that for each $U \in \tau$ that contains a point of a set $K$ for some $K \in \mathcal{K}$, there exists a set $F \in \mathcal{F}$ that satisfies $K \subseteq F \subseteq U$.
\end{definition}

The class of strong $\Sigma$ spaces contains all metrizable spaces, $\sigma$-compact spaces, Lindel\"of $\Sigma$ spaces, and paracompact $\Sigma$ spaces.
Therefore, according to theorem~\ref{strong-simga-is-d}, all these spaces are D-spaces.
Buzyakova also tried but failed to answer the question whether each (regular) Lindel\"of space is a D-space.
However, she mentioned one of the approaches to solve this problem as proposed by a question of Arhangel'skii.

\begin{question}[Arhangel'skii]\label{arhangelskii-question}
	Is it true that a continuous image of a (regular) Lindel\"of D-space is a D-space?
\end{question}

A generalization of the property D, the property ``aD'' was also introduced in the paper Arhangel'skii and Buzyakova (2002)$\,^{\text{\cite{arhangelskii-buzyakova}}}$.
The weaker property aD is defined as follows.

\begin{definition}[aD-space]
	A topological space $(X, \tau)$ is an aD-space if for every closed set $F$ and every open cover $\mathcal{O}$ of $X$, there exists a discrete subset $A$ of $F$ and an \linebreak assignment ${ C : A \rightarrow \mathcal{O} }$ that satisfies $a \in C(a)$ for all $a$ in $A$, where the family \linebreak ${ C(A) = \{ C(a) : a \in A \} }$ is a cover of $F$.
\end{definition}

They proved the following facts about aD-spaces.

\begin{theorem}
	Every paracompact space is an aD-space.
\end{theorem}

\begin{theorem}\label{addition-theorem-ad-space}
	Let $(X, \tau)$ be a regular topological space such that $X = Y \cup Z$, where $Y$ is a paracompact subspace of $X$ and $Z$ is an aD space, then $X$ is an aD space.
\end{theorem}

\begin{theorem}
	Every closed subspace of an aD-space is an aD-space.
\end{theorem}

\begin{theorem}\label{finite-union-of-paracompact-is-ad}
	If a regular topological space $(X, \tau)$ is the union of a finite collection of paracompact subspaces, then $X$ is an aD-space.
\end{theorem}

Theorems about unions of subspaces such as theorem~\ref{addition-theorem-ad-space} are commonly called addition theorems.
For more on the property aD, the reader should consult the paper already mentioned where the property was first introduced, Arhangel'skii (2004)$\,^{\text{\cite{arhangelskii-1}}}$, and Arhangel'skii (2005)$\,^{\text{\cite{arhangelskii-2}}}$.

Important cardinals associated with D-spaces are the \textit{extent} and \textit{Lindel\"of degree} of a space.
The \textit{extent} of a topological space $(X,\tau)$, denoted as $e(X)$, is the supremum of all the cardinalities of its closed discrete subsets.
On the other hand, the Lindel\"of degree of $X$, denoted as $L(X)$, is the least cardinal $\kappa$ such that every open cover of $X$ has a subcover of cardinality less than or equal to $\kappa$.

For any topological space $(X, \tau)$ it is always true that $e(X) \leq L(X)$.
To prove this, take any closed discrete set $F$ of $X$.
Since for each $x \in F$, one can find an open set $U_x$ such that $U_x \cap F = \{x\}$, then the family $\mathcal{O} = \{ U_x : x \in F \} \cup \{ X \setminus F \}$ is an open cover of $X$.
Furthermore, any subcover of $\mathcal{O}$ must contain all the sets in $\{ U_x : x \in F \}$ since no point of $F$ is contained in $X \setminus F$ and no two sets $U_x$ and $U_y$ where $x \neq y$ contain a common point in $F$.
Thus, $\mathcal{O}$ has a minimal subcover $\mathcal{O}^{*}$ and its cardinality is either $|F|$ or $|F| + 1$ depending on whether the family $\{ U_x : x \in F \}$ alone covers $X$ or not.
Therefore, $L(X)$ is at least $|F|$.
However, this is true for every closed set $F$.
Hence, $L(X)$ is an upper bound on all the cardinalities of the closed discrete sets of $X$, i.e.\ $L(X) \geq e(X)$.

In the case of D-spaces, the extent and Lindel\"of degree coincide, i.e.\ if $(X, \tau)$ is a D-space, then $e(X) = L(X)$.
This is since any open cover could be converted to an open neighborhood assignment, i.e.\ if $\mathcal{O} = \{U_a : a \in A\}$ is an open cover, then for each $x \in X$ there is a point $a_x$ in $A$ such that $x \in U_{a_x}$.
Define $N:X \rightarrow \tau$ as $N(x) = U_{a_x}$, then $N$ is an open neighborhood assignment.
Since $X$ is a D-space, there exists a closed discrete set $D$ such that $\bigcup N(D) = X$.
This implies that $\mathcal{O}^{*} = \{U_{a_x} : x \in D \}$ is a subcover of $\mathcal{O}$ with cardinality $|D|$.
However, $|D| \leq e(X)$, i.e.\ $|\mathcal{O}^{*}| \leq e(X)$.
This is true for any open cover $\mathcal{O}$, and hence $e(X)$ is an upper bound on the cardinality of subcovers of open covers so $L(X) \leq e(X)$ which implies $L(X) = e(X)$.
For more results on these cardinals, the reader should consult Gruenhage (2011)$\,^{\text{\cite{gruenhage}}}$.

A slightly different definition of the extent of a space was used in Arhangel'skii and Buzyakova (2002)$\,^{\text{\cite{arhangelskii-buzyakova}}}$.
They defined the extent of the space to be the supremum of all the cardinalities of discrete subsets of $X$.
According to this definition, every aD-space with a countable extent is a Lindel\"of space.
A direct consequence of theorem~\ref{finite-union-of-paracompact-is-ad} is stated as follows.

\begin{corollary}
	Let $(X, \tau)$ be a regular topological space of countable extent, i.e.\ $e(X) = \omega_0$.
	If $X$ is the union of a finite family of paracompact spaces, then $X$ is Lindel\"of.
\end{corollary}

Buzyakova, again, proved another remarkable result in Buzyakova (2004)$\,^{\text{\cite{buzyakova-2}}}$ which states the following.


\begin{theorem}
	For every compact topological space $(X, \tau)$, the space $C_p(X)$ of all real valued continuous functions with the topology of pointwise convergence is hereditary D, that is, every subspace of $C_p(X)$ is a D-space.
\end{theorem}

More research was conducted on D-spaces and its generalizations afterwards.
An excellent and extensive survey on D-spaces could be found in Gruenhage (2011)$\,^{\text{\cite{gruenhage}}}$.
Some important relatively new results that are not found in Gruenhage (2011)$\,^{\text{\cite{gruenhage}}}$ could be found in Peng (2012)$\,^{\text{\cite{peng-1}}}$, Peng (2017)$\,^{\text{\cite{peng-2}}}$, Soukup and Szeptycki (2012)$\,^{\text{\cite{soukup-szeptycki}}}$, and Zhang and Shi (2012)$\,^{\text{\cite{zhang-shi}}}$.

In the paper Soukup and Szeptycki (2012)$\,^{\text{\cite{soukup-szeptycki}}}$, an example of a Hausdorff Lindel\"of space that doesn't have the property D was constructed by introducing the $\Diamond$ (diamond) axiom/principle which implies the continuum hypothesis, i.e.\ $2^{\omega_0} = \omega_1$ and which is independent of ZFC. For more on $\Diamond$, the reader can consult Jech (2002)$\,^{\text{\cite{jech}}}$.
Still, there is no counterexample of a regular Lindel\"of space which is not a D-space (in ZFC with or without $\Diamond$).

\section{A discussion on two open problems}\label{sec:a-discussion-on-two-open-problems}

We will discuss two problems central to this thesis which are:

\begin{question}\label{qstn:is-lindelof-d}
	Is every (regular) (hereditary) Lindel\"of space a D-space?
\end{question}

\begin{question}
	Is every (regular) (hereditary) paracompact space a D-space?
\end{question}

As mentioned in the previous section, these two questions remain unanswered to this day.
Many attempted to solve the first problem and there is still no clear insight on how one can prove it or rather construct a counterexample.
An approach to the first problem could start in the following manner.

Assume that $(X, \tau)$ is a ($T_1$) Lindel\"of space and let $N:X \rightarrow \tau$ be any open neighborhood assignment.
We will construct a set $K = \{ x_\alpha : \alpha < \gamma \}$ for some ordinal $\gamma$ with the property that $x_\beta \notin N(x_\alpha)$ for all $\alpha < \beta < \gamma$ by transfinite recursion.
At the initial step $0$, we take any point $x_0$ in $X$ and let $K_0 = \{ x_0 \}$.
Now, recursively define $x_\xi$ by taking any point $x_\xi$ in $X \setminus \bigcup\limits_{\alpha < \xi}N(x_\alpha)$.
There is no need to separate the successor step from the limit step of the recursion since they are both the same.
Define $K_\xi$ as $\{x_\alpha : \alpha < \xi\}$, then from the way we constructed $x_\xi$ so that $x_\alpha \neq x_\beta$ for all $\alpha < \beta < \xi$, it is then clear that $K_\alpha \subsetneq K_\beta$ for all $\alpha < \beta < \xi$.
This recursion must end at some ordinal $\gamma$ since $K_\alpha \subseteq X$ and each $K_\alpha$ is different for all $\alpha$.
If it didn't end, then there exists an order isomorphism $\mathcal{C} : \text{\textbf{Ord}} \rightarrow \mathcal{K}$ where $\text{\textbf{Ord}}$ is the class of all ordinals and $\mathcal{K}$ is $\{ K_\alpha: \alpha \in \text{\textbf{Ord}} \}$ ordered by set inclusion $\subseteq$.
This implies that $\mathcal{K}$ is a proper class.
However, $\mathcal{K} \subseteq \mathcal{P}(X)$, and thus $\mathcal{P}(X)$ is a proper class too which contradicts the power set axiom.

Moreover, from the way $K = K_\gamma$ is constructed, we can define a total order relation $\preceq$ on $K$ such that $x_\alpha \preceq x_\beta$ for all $\alpha \leq \beta < \gamma$.
Furthermore, for all $\xi \geq \gamma$, we have $K_\xi = K_\gamma = K$ which means that $X \setminus \bigcup\limits_{\alpha < \xi} N(x_\alpha) = \phi$.
Thus, the family $\{ N(x_\alpha) : \alpha < \gamma\}$ is an open cover of $X$ which by the Lindel\"of property implies that it reduces to a countable subcover $\{N(x): x \in D\}$ where $D \subseteq K$.
What can we say about $D$? Obviously, $D$ is countable.
The mistake we mentioned before in  Borges and Wehrly (1996)$\,^{\text{\cite{borges-wehrly-qa}}}$ was that they had assumed $D$ is order isomorphic to the integers $\mathbb{N}$, i.e.\ $D = \{ x_i : i \in \mathbb{N} \}$ and $x_1 \prec x_2 \prec \dots$ since $D$ is also ordered by $\preceq$.
If so, then the set $D$ becomes a closed discrete subset of $X$ such that $\bigcup N(D) = \bigcup\limits_{d \in D} N(d) = X$.
To prove this take any $x_i$ in $D$, then from the way we constructed each $x_i$, it is clear that $x_{i + j} \notin N(x_i)$ for all $j \in \mathbb{N}$; hence, $U_i = N(x_i) \setminus \{ x_1, \dots, x_{i - 1} \}$ is an open neighborhood of $x_i$ (because $X$ is $T_1$) such that $U_i \cap D = \{ x_i \}$.
This means that $X$ is a D-space.
So the main problem here is obtaining a countable subcover of the constructed set $K$ that is order isomorphic to the integers, i.e.\ we don't want to have $x_i \prec x_j$ for some $i > j$, e.g.\ we don't want to have $x_3 \prec x_1$.

Another possible approach to the first question is by answering question~\ref{arhangelskii-question} of Arhangel'skii.
Since this approach was discussed in a private communication, we don't know how Arhangel'skii sees that the solution of question~\ref{arhangelskii-question} could lead to a solution of question~\ref{qstn:is-lindelof-d}.
Our assumption to what he had in mind is that given any regular Lindel\"of space $(Y, \tau_Y)$, if it possible to construct another space $(X, \tau_X)$ such that $(X, \tau_X)$ is both Lindel\"of and D, and if it there is a continuous map from $(X, \tau_X)$ onto $(Y, \tau_Y)$, then as a consequence, $(Y, \tau_Y)$ becomes a D-space (if it were true that a continuous image of a Lindel\"of D-space is a D-space).

Regarding the second question, a possible approach as proposed by Gruenhage \linebreak (2011)$\,^{\text{\cite{gruenhage}}}$, is to construct for each open neighborhood assignment $N: X \rightarrow \tau$ a locally finite open refinement $\mathcal{O}_r$ of $\mathcal{O} = \{ N(x) : x \in X \}$ such that for each $U$ in $\mathcal{O}_r$ there is a point $x_U$ in $X$ for which $x_U \in U \subseteq N(x_U)$.
If this construction is possible, then $X$ is immediately a paracompact D-space.
To prove this, let $D = \{ x_U : U \in \mathcal{O}_r \}$.
It is clear that $\bigcup N(D) = \bigcup\limits_{d \in D} N(d) = X$ since $\mathcal{O}_r$ is an open cover of $X$ and $U \subseteq N(x_U)$ for all $U$.
It remains to show that $D$ is closed and discrete in $X$.
Take any point $x_{U_0}$ in $D$, then from the fact that $\mathcal{O}_r$ is a locally finite, there exists an open neighborhood $V_{U_0}$ of $x_{U_0}$ such that $V_{U_0}$ intersects at most finitely many open sets in $\mathcal{O}_r$.
Hence, $V_{U_0}$ contains at most finitely many points from $D$, say $x_{U_1}$, \dots, $x_{U_n}$.
Therefore, $V_{U_0}^{*} = V_{U_0} \setminus \left( \{ x_{U_1}, \dots, x_{U_n} \} \setminus \{ x_{U_0} \} \right)$ is an open neighborhood of $x_{U_0}$ which intersects $D$ in only $x_{U_0}$.
Since $x_{U_0}$ was chosen arbitrarily, then $D$ is a discrete set.
Similarly, to prove that $D$ is closed, take any $x_0$ in $X \setminus D$.
There exists an open neighborhood $V_0$ of $x_0$ such that $V_0$ intersects at most finitely many open sets in $\mathcal{O}_r$, and hence contains at most contains at most finitely many points from $D$, say $x_{U_1}$, \dots, $x_{U_n}$.
We can ensure that $x_0 \notin \{ x_{U_1}, \dots, x_{U_n} \}$ since $x_0 \in X \setminus D$, and thus $V_{0}^{*} = V_{0} \setminus \left( \{ x_{U_1}, \dots, x_{U_n} \} \setminus \{ x_0 \} \right)$ is an open neighborhood of $x_0$ such that $V_{0}^{*} \subseteq X \setminus D$.
This ends the proof.

As a note here, if it is possible to prove that every regular paracompact space is a D-space, then a direct consequence of this is that every regular Lindel\"of space is also a D-space (see Engelking, 1989)$\,^{\text{\cite{engelking}}}$.
This is true since every regular Lindel\"of space is actually a paracompact space.
However, it seems harder to show whether every regular paracompact space is a D-space.

The previous discussion was concerning the approaches to the affirmative answer of each of the two questions.
However, it is not yet clear how one should go about constructing a counterexample for the negative answer of each question.
But we know that if the answer to the first question is no, then the answer is no as well to the second question.
Perhaps, the current axioms of Set Theory (ZFC) are insufficient for such constructions.
Furthermore, since both regular Lindel\"of and regular paracompact implies normal, then the separation property is already strong and introducing a stronger separation property may not help.
In other words, even if a proof was provided for a stronger separation property, it will still be interesting to know the answer when the separation property is regular/normal.

In chapter~\ref{ch:a-new-approach}, we will introduce a new approach that might give us a better insight into these problems.
The key theme in our approach is our emphasis on studying continuous maps that are related to covering properties in general, and to D-spaces in particular.
We start by introducing generalizations to the neighborhood assignments $N : X \rightarrow \tau$ and then obtain a special continuous map, that we call, the companion map using a specific topology, that we call, the principal ultrafilter topology (or puf topology for short).
We then study some categorical constructions in \textbf{Top} related to these maps and to the puf topology.

In chapter~\ref{ch:applications-to-the-theory-of-d-spaces}, we will use companion maps to obtain an equivalent definition of D-spaces that is central to this thesis.
After that, we will discuss our attempts in solving question~\ref{qstn:is-lindelof-d} and obtain other results that could help in constructing counterexamples.
Finally, we will apply companion maps to see if we can find connections between D-spaces and some other covering properties like paracompactness and metacompactness.

\clearpage
\chapter{A new approach}\label{ch:a-new-approach}
\section{Preliminaries}\label{sec:preliminaries}

Observe that for a topological space $(X, \tau)$, any open neighborhood assignment $N : X \rightarrow \tau$ could be seen as an open cover indexed by the set $X$, i.e.\ as $\{ U_x = N(x) : x \in X \}$.
This family is indeed an open cover since every point is covered by a neighborhood, and thus the entire space is covered.
Conversely, any open cover of the form $\{U_x : x \in X \}$ satisfying $x \in U_x$ for all $x$ in $X$ could be seen as an open neighborhood assignment.
Generally speaking, an open cover is a family of open sets of $(X, \tau)$ that covers $X$ and which is indexed by some set $A$; thus, it has the form $\mathcal{O} = \{ U_a : a \in A \}$.
By analogy, we can construct a function $C : A \rightarrow \tau$ such that $C(a) = U_a$.
We will call this function an \textit{open covering assignment} and it must satisfy $\bigcup C(A) = \bigcup\limits_{a
\in A}{C(a)} = X$.
In general, a function ${ O : A \rightarrow \tau }$ that doesn't necessarily satisfy $\bigcup O(A) = X$ will be called an \textit{open set assignment}.
To study the relations between covering properties and the property D, we will introduce a new approach utilizing special continuous topological mappings;
however, we must first introduce some important concepts and notation.

\section{The principal ultrafilter topology}\label{sec:the-principal-ultrafilter-topology}

For any set $X$, the set $\mathcal{P}(X)$ together with the relation $\subseteq$ forms a structure known as a \textit{poset}.
A subset $\mathcal{F}$ of $\mathcal{P}(X)$ is called a \textit{filter} (Jech 2002)$\,^{\text{\cite{jech}}}$ if it has the following properties:


\begin{enumerate}
	\item $\mathcal{F}$ is nonempty.
	\item For any sets $A$ and $B$ in $\mathcal{F}$, there is a set $C$ in $\mathcal{F}$ such that $C \subseteq A$ and $C \subseteq B$.
	\item For every set $A$ in $\mathcal{F}$ and every set $B$ in $\mathcal{P}(X)$, $A \subseteq B$ implies that $B$ is in $\mathcal{F}$ too.
\end{enumerate}

An \textit{ultrafilter} on a poset $(\mathcal{P}(X), \subseteq)$ is a maximal filter on this poset, that is, there is no filter $\mathcal{F}\,'$ such that $\mathcal{F} \subsetneq \mathcal{F}\,' \subsetneq \mathcal{P}(X)$.
A \textit{principal filter} generated by a set $A$ on $\mathcal{P}(X)$ is a filter constructed by taking all sets in $\mathcal{P}(X)$ containing $A$, i.e.\ the set $\{B \in \mathcal{P}(X) : A \subseteq B \}$.
A \textit{principal ultrafilter} on a poset $(\mathcal{P}(X), \subseteq)$ is a maximal principal filter on this poset;
thus, any principal ultrafilter in this poset is a principal filter generated by a singleton $\{x\}$ for some $x$ in $X$.
Hence, for every point $x$ in $X$, we can obtain the principal ultrafilter
\[
\mathcal{U}(x) = \{ B \in \mathcal{P}(X) : \{ x \} \subseteq B \}
\]
where $\mathcal{U}(x)$ will be used to denote the principal ultrafilter generated by the singleton $\{x\}$ of the point $x$ in $X$ and unless the set $X$ is known from the context, we will use the notation $\mathcal{U}_{X}(x)$ otherwise.
Therefore, for any set $X$, we obtain a set function
${ \mathcal{S}: X \rightarrow \mathcal{P} \left( \mathcal{P}(X) \right) }$ defined as $\mathcal{S}(x) = \mathcal{U}(x)$, i.e.\ takes each point of $X$ to its principal ultrafilter on ${ (\mathcal{P}(X), \subseteq) }$.
Similarly, unless the set $X$ is known from the context, we will use the notation ${ \mathcal{S}_{X}(x) = \mathcal{U}_{X}(x) }$.
The image of this function, i.e.\ $\mathcal{S}(X)$ is the set of all principal ultrafilters on the poset $(\mathcal{P}(X), \subseteq)$.
This brings us to the following definition.

\begin{definition}
	Let $X$ be any nonempty set, the \textbf{principal ultrafilter topology} (abbreviated as ``puf topology'') on $\mathcal{P}(X)$ is the topology generated by the subbase $\mathcal{S}(X)$.
	This topology will be denoted as $\tau_{\mathcal{S}(X)}$ or just as $\tau_{\mathcal{S}}$ when $X$ is known from the context, i.e.\ we have the topological space $(\mathcal{P}(X), \tau_{\mathcal{S}})$.
\end{definition}


\section{The companion map}\label{sec:the-companion-map}

In the following proposition, we introduce a special map that will be used frequently later on.

\begin{proposition}\label{open-set-assignment-to-continuous-map}
	Let $(X, \tau)$ be a topological space and $A$ be any set.
	A function \linebreak ${ O : A \rightarrow \mathcal{P}(X) }$ is an open set assignment iff there exists a continuous map
	\[
		f : (X, \tau) \rightarrow \left( \mathcal{P}(A), \tau_{\mathcal{S}(A)} \right)
	\]
	satisfying $f^{-1}(\mathcal{U}_{A}(a)) = O(a)$ for all $a$ in $A$.
\end{proposition}

\begin{proof}
	Assume that $O : A \rightarrow \mathcal{P}(X)$ is an open set assignment, then $O(a) \in \tau$ for all $a$ in $A$.
	Define a function $f : X \rightarrow \mathcal{P}(A)$ as
	\[
	f(x) = \{ a \in A : x \in O(a) \}
	\]
	i.e.\ it takes each point to the set of all indices of the open sets in which the point lies.
	Moreover, each set of indices lies in every principal ultrafilter generated by each index it contains.
	Since every subbasic open set in $(\mathcal{P}(A), \tau_{\mathcal{S}})$ is of the form $\mathcal{U}(a)$ for some $a \in A$, then to prove that $f$ is a continuous map it suffices to show $f^{-1}(\mathcal{U}(a)) = O(a)$.
	Consider the following,
	\begin{align*}
    	\begin{aligned}
			x \in O(a) & \iff a \in f(x) \\
			& \iff f(x) \in \mathcal{U}(a) \\
			& \iff x \in f^{-1}(\mathcal{U}(a))
      	\end{aligned}
	\end{align*}
	which proves the sufficiency part.
	Conversely, assume that there exists a continuous map $f : (X, \tau) \rightarrow (\mathcal{P}(A), \tau_{\mathcal{S}})$ satisfying $f^{-1}(\mathcal{U}(a)) = O(a)$.
	Since $f$ is continuous and $\mathcal{U}(a) \in \tau_{\mathcal{S}}$ for all $a$ in $A$, then $O(a) = f^{-1}(\mathcal{U}(a)) \in \tau$ for all $a$ in $A$.
	Hence, $O$ is an open set assignment.
\end{proof}


\begin{proposition}\label{open-covering-assignment-to-continuous-map}
	Let $(X, \tau)$ be a topological space and $A$ be any set.
	A function \linebreak ${ C : A \rightarrow \mathcal{P}(X) }$ is an open covering assignment iff there exists a continuous map
	\[
		f : (X, \tau) \rightarrow \left( \mathcal{P}(A), \tau_{\mathcal{S}(A)} \right)
	\]
	satisfying
	\begin{enumerate}
		\item $f^{-1}(\mathcal{U}_{A}(a)) = C(a)$ for all $a$ in $A$ \label{open-covering-assignment-to-continuous-map-1}
		\item $f^{-1}\left(\bigcup \mathcal{S}(A)\right) = X \iff f^{-1}(\{ \phi \}) = \phi$ \label{open-covering-assignment-to-continuous-map-2}
	\end{enumerate}
\end{proposition}

\begin{proof}
	Assume that $C : A \rightarrow \mathcal{P}(X)$ is an open covering assignment, then $C(a) \in \tau$ for all $a$ in $A$ and $\bigcup C(A) = X$.
	Since $C$ is also an open set assignment, then as a consequence of proposition~\ref{open-set-assignment-to-continuous-map}, there exists a continuous map $f : (X, \tau) \rightarrow \left( \mathcal{P}(A), \tau_{\mathcal{S}} \right)$ satisfying~\ref{open-covering-assignment-to-continuous-map-1}.
	Moreover,
	\begin{align*}
		\begin{aligned}
			X = \bigcup\limits_{a \in A}{C(a)} & = \bigcup\limits_{a \in A} f^{-1}(\mathcal{U}(a)) \\
			& = f^{-1} \left( \bigcup\limits_{a \in A}{\mathcal{U}(a)} \right) \\
			& = f^{-1} \left( \bigcup \mathcal{S}(A) \right)
		\end{aligned}
	\end{align*}
	Hence,~\ref{open-covering-assignment-to-continuous-map-2} follows.
	Conversely, assume that there exists a continuous map \linebreak ${ f : (X, \tau) \rightarrow \left( \mathcal{P}(A), \tau_{\mathcal{S}} \right) }$ satisfying both~\ref{open-covering-assignment-to-continuous-map-1} and~\ref{open-covering-assignment-to-continuous-map-2}.
	Again, as a consequence of proposition~\ref{open-set-assignment-to-continuous-map}, $C$ is an open set assignment.
	It remains to show that $\bigcup C(A) = X$.
	Consider the following,
	\begin{align*}
		\begin{aligned}
			X = f^{-1} \left( \bigcup \mathcal{S}(A) \right) & = f^{-1} \left( \bigcup\limits_{a \in A} \mathcal{U}(a) \right) \\
			& = \bigcup\limits_{a \in A} f^{-1} \left( \mathcal{U}(a) \right) \\
			& = \bigcup\limits_{a \in A} C(a)
		\end{aligned}
	\end{align*}
	Therefore, $C$ is an open covering assignment.
\end{proof}


\begin{proposition}\label{neighborhood-assignment-to-continuous-map}
	Let $(X, \tau)$ be a topological space and $A$ be any set.
	A function \linebreak ${ N : A \rightarrow \mathcal{P}(X) }$ is an open neighborhood assignment iff there exists a continuous map
	\[
		f : (X, \tau) \rightarrow \left( \mathcal{P}(X), \tau_{\mathcal{S}(X)} \right)
	\]
	satisfying
	\begin{enumerate}
		\item $f^{-1}\left(\mathcal{U}_{X}(x)\right) = N(x)$ for all $x$ in $X$ \label{neighborhood-assignment-to-continuous-map-1}
		\item $f(x) \in \mathcal{U}_{X}(x)$ for all $x$ in $X$ \label{neighborhood-assignment-to-continuous-map-2}
	\end{enumerate}
\end{proposition}

\begin{proof}
	Assume that $N : A \rightarrow \mathcal{P}(X)$ is an open neighborhood assignment, then it is an open covering assignment and an open set assignment with the domain $X$; thus, as a consequence of proposition~\ref{open-set-assignment-to-continuous-map} there exists a continuous map $f : (X, \tau) \rightarrow \left( \mathcal{P}(X), \tau_{\mathcal{S}} \right)$ satisfying~\ref{neighborhood-assignment-to-continuous-map-1}.
	To prove~\ref{neighborhood-assignment-to-continuous-map-2}, we will use the fact that $x \in N(x)$ for all $x$ in $X$; hence by~\ref{neighborhood-assignment-to-continuous-map-1}, $x \in f^{-1}(\mathcal{U}(x))$.
	Therefore, $f(x) \in \mathcal{U}(x)$.
	Conversely, assume that there exists a continuous map ${ f : (X, \tau) \rightarrow \left( \mathcal{P}(X), \tau_{\mathcal{S}} \right) }$ satisfying both~\ref{neighborhood-assignment-to-continuous-map-1} and~\ref{neighborhood-assignment-to-continuous-map-2}.
	Again, as a consequence of proposition~\ref{open-set-assignment-to-continuous-map}, $N$ is an open set assignment.
	It remains to show that $x \in N(x)$ for all $x$ in $X$.
	From~\ref{neighborhood-assignment-to-continuous-map-2}, $x \in f^{-1}(\mathcal{U}(x))$ for all $x$ in $X$; thus, from~\ref{neighborhood-assignment-to-continuous-map-1}, we conclude that $N$ is an open neighborhood assignment.
\end{proof}

\begin{corollary}\label{open-set-assignment-to-continuous-map-is-unique}
	Let $(X, \tau)$ be a topological space, $A$ be any set, and ${ O : A \rightarrow \tau }$ be an open set assignment.
	There exists a \textbf{unique} continuous map
	\[
	f : (X, \tau) \rightarrow \left( \mathcal{P}(A), \tau_{\mathcal{S}(A)} \right)
	\]
	satisfying $f^{-1}\left(\mathcal{U}_{A}(a)\right) = O(a)$ for all $a$ in $A$.
\end{corollary}

\begin{proof}
	As a direct consequence of proposition~\ref{open-set-assignment-to-continuous-map}, a continuous map \linebreak ${ f : (X, \tau) \rightarrow \left( \mathcal{P}(A), \tau_{\mathcal{S}} \right) }$ exists which satisfies $f^{-1}(\mathcal{U}(a)) = O(a)$ for all $a$ in $A$.
	\linebreak To prove its uniqueness, assume that $f$ is not unique, i.e.\ assume that there exists a different continuous map $g : (X, \tau) \rightarrow \left( \mathcal{P}(A), \tau_{\mathcal{S}} \right)$ satisfying $g^{-1}(\mathcal{U}(a)) = O(a)$ for all $a$ in $A$.
	We mean by $g$ being a different map from $f$ that there exists some $x$ in $X$ such that $g(x) \neq f(x)$.
	Assume without loss of generality that there exists some $a$ in $A$ such that $a \in g(x)$ and $a \notin f(x)$.
	Hence, $g(x) \in \mathcal{U}(a)$ and $f(x) \notin \mathcal{U}(a)$ which implies that $x \in g^{-1}(\mathcal{U}(a))$ and $x \notin f^{-1}(\mathcal{U}(a))$.
	However, $g^{-1}(\mathcal{U}(a)) = O(a) = f^{-1}(\mathcal{U}(a))$, a contradiction.
\end{proof}

Corollary~\ref{open-set-assignment-to-continuous-map-is-unique} implies that the continuous map satisfying the necessary condition in proposition~\ref{open-covering-assignment-to-continuous-map} is also unique.
Similarly, the continuous map satisfying the necessary condition in proposition~\ref{neighborhood-assignment-to-continuous-map} is unique too.
The continuous map that is associated with any one of these assignments (open set assignment, open covering assignment, open neighborhood assignment) will be called the \textit{companion map} of the associated assignment.
According to this result and according to the argument of proposition~\ref{open-set-assignment-to-continuous-map}, the map $f : (X, \tau) \rightarrow (\mathcal{P}(A), \tau_{\mathcal{S}})$ defined as $f(x) = \{ a \in A : x \in O(a) \}$ is the companion map of any open set assignment $O : A \rightarrow \tau$.

\section{Subfamilies and refinements}\label{sec:subfamilies-and-refinements}

Topological properties that are related to covers of a space (usually open covers) are referred to as covering properties.
Covering properties imply the existence of certain subcovers or refinements of open covers.
Analogously, for companion maps we will obtain similar notions to subcovers and refinements of open covers as follows.

\begin{proposition}\label{subfamily-map}
	Let $(X, \tau)$ be a topological space and $O : A \rightarrow \tau$ be an open set assignment.
	Suppose that $f : (X, \tau) \rightarrow (\mathcal{P}(A), \tau_{\mathcal{S}(A)})$ is the companion map of $O$, then for any set $D \subseteq A$, the restriction of the open set assignment $O$ to $D$ (i.e.\ $O\restriction_{D} : D \rightarrow \tau$) is associated with the companion map $g : (X, \tau) \rightarrow (\mathcal{P}(D), \tau_{\mathcal{S}(D)})$ defined as $g(x) = f(x) \cap D$.
\end{proposition}

\begin{proof}
	Any restriction of an open set assignment is an open set assignment.
	We need to prove that $g^{-1}(\mathcal{U}_{D}(d)) = O\restriction_{D}(d)$ for all $d$ in $D$.
	Since $f^{-1}(\mathcal{U}_{X}(d)) = O(d) = O\restriction_{D}(d)$ for all $d$ in $D$, then it suffices to show that $g^{-1}(\mathcal{U}_{D}(d)) = f^{-1}(\mathcal{U}_{X}(d))$.
	Let $d \in D$, then consider the following

	\begin{align*}
    	\begin{aligned}
			x \in g^{-1}(\mathcal{U}_{D}(d)) & \iff g(x) \in \mathcal{U}_{D}(d) \\
			& \iff d \in g(x) = f(x) \cap D \\
			& \iff d \in f(x) \\
			& \iff f(x) \in \mathcal{U}_{X}(d) \\
			& \iff x \in f^{-1}(\mathcal{U}_{X}(d)) \\
      	\end{aligned}
	\end{align*}
	Hence, by proposition~\ref{open-set-assignment-to-continuous-map} and corollary~\ref{open-set-assignment-to-continuous-map-is-unique}, the map $g$ is the unique companion map of the open set assignment $O\restriction_{D}$.
\end{proof}

Since any open set assignment corresponds to a family of open sets, then the restriction of an open set assignment to a subset of its domain corresponds to a subfamily of the family of open sets induced by that open set assignment.

\begin{corollary}\label{subcover-map}
	Let $(X, \tau)$ be a topological space and $C : A \rightarrow \tau$ be an open covering assignment.
	Suppose that $f : (X, \tau) \rightarrow (\mathcal{P}(A), \tau_{\mathcal{S}(A)})$ is the companion map of $C$, then for any set $D \subseteq A$ the assignment $C\restriction_{D}$ is also an open covering assignment iff there exists a continuous map $g : (X, \tau) \rightarrow (\mathcal{P}(D), \tau_{\mathcal{S}(D)})$ satisfying
	\begin{enumerate}
		\item $g(x) = f(x) \cap D$ for all $x$ in $X$ \label{subcover-map-1}
		\item $g^{-1}(\{ \phi \}) = \phi$ \label{subcover-map-2}.
	\end{enumerate}
\end{corollary}

\begin{proof}
	Assume that the $C\restriction_{D}$ is an open covering assignment, then by proposition~\ref{subfamily-map} there exists a continuous map $g : (X, \tau) \rightarrow (\mathcal{P}(D), \tau_{\mathcal{S}})$ satisfying~\ref{subcover-map-1} which is the companion map of $C\restriction_{D}$.
	By propositions~\ref{open-covering-assignment-to-continuous-map} and~\ref{open-set-assignment-to-continuous-map-is-unique}, the companion map $g$ of $C\restriction_{D}$ also satisfies~\ref{subcover-map-2}.
	Conversely, assume that there exists a continuous map  $g : (X, \tau) \rightarrow (\mathcal{P}(D), \tau_{\mathcal{S}})$ satisfying both~\ref{subcover-map-1} and~\ref{subcover-map-2}.
	The assignment $C\restriction_{D}$ is indeed an open set assignment and its companion map is $g$ by proposition~\ref{subfamily-map}.
	Thus, from proposition~\ref{open-covering-assignment-to-continuous-map}, again, the assignment $C\restriction_{D}$ is an open covering assignment.
\end{proof}

Note that for any open covering assignment $C : A \rightarrow \tau$, if the restriction of $C$ to a set $D \subseteq A$ is also an open covering assignment, then it implies that the corresponding open cover of $C\restriction_{D}$ is a subcover of the open cover induced by $C$.
Since open neighborhood assignments are also open covering assignments, then corollary~\ref{subcover-map} could also be applied to them.
Moreover, for an open neighborhood assignment $N : X \rightarrow \tau$, any set $D \subseteq X$ for which the restriction of $N$ to $D$ is an open covering assignment will be called a \textit{kernel} of $N$.
In this case, any kernel $D \subseteq X$ of $N$ satisfies $\bigcup N(D) = \bigcup\limits_{d \in D}{N(d)} = X$.

\begin{definition}
	Let $(X, \tau)$ be a topological space and $O : A \rightarrow \tau$ be an open set assignment.
	A \textbf{refinement} of $O$ is another open set assignment $O^{*} : A \rightarrow \tau$ satisfying $O^{*}(a) \subseteq O(a)$ for all $a \in A$.
\end{definition}

As an important note here, a refinement of an assignment has a different meaning than that of an open refinement of an open cover.
An open refinement is more general since it could contain more than one open subset for a single open set in the open cover.
Moreover, an open refinement could also have a cardinality less than that of the original open cover.

\begin{proposition}\label{refinement-companion-map}
    Let $(X, \tau)$ be a topological space and $O : A \rightarrow \tau$ be an open set assignment, then an open set assignment $O^{*} : A \rightarrow \tau$ is a refinement of $O$ iff $f_{O}$ was the companion map of $O$ and $f_{O^{*}}$ was the companion map of $O^{*}$ with $f_{O^{*}}(x) \subseteq f_{O}(x)$ for all $x$ in $X$.
\end{proposition}

\begin{proof}
	Assume that $O^{*}$ is a refinement of $O$, $f_{O}$ is the companion map of $O$, and that $f_{O^{*}}$ is the companion map of $O^{*}$.
	Take any $x \in X$ and any $a \in f_{O^{*}}(x)$, i.e.\ $f_{O^{*}}(x) \in \mathcal{U}(a)$.
	Since $f_{O^{*}}^{-1}(\mathcal{U}(a)) = O^{*}(a)$, we have $x \in f_{O^{*}}^{-1}(f_{O^{*}}(x)) \subseteq f_{O^{*}}^{-1}(\mathcal{U}(a)) = O^{*}(a)$.
	Since $O^{*}$ is a refinement of $O$, then $O^{*}(a) \subseteq O(a)$; hence, $x \in O(a)$.
	Since $f_{O}^{-1}(\mathcal{U}(a)) = O(a)$, we have $f_{O}(x) \in \mathcal{U}(a)$.
	Thus, $a \in f_O(x)$ which ends the proof.
	Conversely, assume that $f_{O}$ is the companion map of $O$ and $f_{O^{*}}$ is the companion map of $O^{*}$ and that $f_{O^{*}}(x) \subseteq f_{O}(x)$ for all $x$ in $X$.
	We need to show that $O^{*}(a) \subseteq O(a)$ for all $a$ in $A$.
	Take any $a \in A$ and any $x \in O^{*}(a)$.
	Hence, $x \in f_{O^{*}}^{-1}(\mathcal{U}(a))$, i.e.\ $f_{O^{*}}(x) \in \mathcal{U}(a)$, equivalently $a \in f_{O^{*}}(x)$.
	Thus, $a \in f_{O}(x)$ which implies that $f_{O}(x) \in \mathcal{U}(a)$, i.e.\ $x \in f_{O}^{-1}(\mathcal{U}(a)) = O(a)$.
\end{proof}

For the companion map $f : (X, \tau) \rightarrow (\mathcal{P}(A), \tau_{\mathcal{S}(A)})$ of an open set assignment \linebreak ${ O : A \rightarrow \tau }$, we will call any map $f_{*} : (X, \tau) \rightarrow (\mathcal{P}(A), \tau_{\mathcal{S}(A)})$ satisfying $f_{*}(x) \subseteq f(x)$ a \textit{refinement map} of $f$.
Additionally, for any neighborhood assignment $N : X \rightarrow \tau$ if $N^{*} : X \rightarrow \tau$ is a refinement of $N$ such that $x \in N^{*}(x)$ for all $x$ in $X$, then we will call $N^{*}$ a \textit{neighborhood refinement} of $N$ and the companion map $f_{N^{*}}$ of $N^{*}$ will be called a \textit{neighborhood refinement map} of the companion map $f_{N}$ of $N$.

\section{Related structures in the category \textbf{Top}}\label{sec:related-structures-in-the-categorytextbf}

We will discuss some structures related to the puf topology and to companion maps within the category \textbf{Top} whose objects are topological spaces and whose morphisms are continuous maps.
For more on \textbf{Top} and category theory in general, the reader can look at appendix~\ref{sec:category-theory-and-the-categorytextbf}.

\begin{definition}
	For any topological space $(X, \tau_{X})$, the coslice category $(X, \tau_{X}) \downarrow \text{\textbf{Top}}$ is the category with the following:
	\begin{enumerate}
		\item Objects are all continuous maps $f:(X, \tau_{X}) \rightarrow (Y, \tau_{Y})$ where $(Y, \tau_{Y})$ is any topological space, i.e.\ all continuous maps that have $(X, \tau_{X})$ as their domain.
		\item A morphism from $f_{1}:(X, \tau_{X}) \rightarrow (Y_1, \tau_{Y_1})$ to $f_{2}:(X, \tau_{X}) \rightarrow (Y_2, \tau_{Y_2})$ is a continuous map $g: (Y_1, \tau_{Y_1}) \rightarrow (Y_2, \tau_{Y_2})$ that satisfies $g \circ f_{1} = f_{2}$, i.e.\ that makes the triangle below commute.

		\begin{center}
			\begin{tikzpicture}
				\node (X) {$(X, \tau_{X})$};
				\node (Y1) [below of=X] {$(Y_1, \tau_{Y_1})$};
				\node (Y2) [node distance=3cm, right of=Y1] {$(Y_2, \tau_{Y_2})$};
				\draw[->] (X) to node {$f_{2}$} (Y2);
				\draw[->] (X) to node [swap] {$f_{1}$} (Y1);
				\draw[->] (Y1) to node [swap] {$g$} (Y2);
			  \end{tikzpicture}
		\end{center}
	\end{enumerate}
\end{definition}

\begin{proposition}\label{refinement-map-puf}
	Let $X$ be a set and $(\mathcal{P}(X), \tau_{\mathcal{S}(X)})$ be the principal ultrafilter topology on $\mathcal{P}(X)$, then any map $R:(\mathcal{P}(X), \tau_{\mathcal{S}(X)}) \rightarrow (\mathcal{P}(X), \tau_{\mathcal{S}(X)})$ that satisfies $R(A) \subseteq A$ for all $A$ in $\mathcal{P}(X)$ is a continuous map.
\end{proposition}


\begin{proof}
	We will show that $R^{-1}(\mathcal{U}_{X}(x)) = \mathcal{U}_{X}(x)$ for all $x$ in $X$ which implies that $R$ is a continuous map.
	For $x$ in $X$, consider
	\begin{align*}
		\begin{aligned}
			A \in R^{-1}(\mathcal{U}(x)) & \Longrightarrow R(A) \in \mathcal{U}(x) \\
			& \Longrightarrow x \in R(A) \subseteq A \\
			& \Longrightarrow x \in A \\
			& \Longrightarrow A \in \mathcal{U}(x)
		\end{aligned}
	\end{align*}
	.
	Hence, $R^{-1}(\mathcal{U}(x)) \subseteq \mathcal{U}(x)$.
	Moreover, let $A \in \mathcal{U}(x)$.
	From the fact that ${ R(A) \subseteq A }$, we obtain $R^{-1}(R(A)) \subseteq R^{-1}(A)$; however, $A \in R^{-1}(R(A))$, and thus $A \in R^{-1}(A) \subseteq R^{-1}(\mathcal{U}(x))$.
	Therefore, $\mathcal{U}(x) \subseteq R^{-1}(\mathcal{U}(x))$ which ends the proof.
\end{proof}

\begin{corollary}\label{refinement-map-coslice}
	Let $(X, \tau)$ be a topological space and $A$ be a set.
	For any continuous map $f : (X, \tau) \rightarrow (A, \tau_{\mathcal{S}(A)})$ and any continuous refinement map $f_{*} : (X, \tau) \rightarrow (A, \tau_{\mathcal{S}(A)})$, there exists a continuous map $R : (A, \tau_{\mathcal{S}(A)}) \rightarrow (A, \tau_{\mathcal{S}(A)})$ that makes the triangle below commute, i.e.\ $R$ is a morphism from $f$ to $f_{*}$ in the coslice category $(X, \tau) \downarrow \text{\textbf{Top}}$.
	\begin{center}
		\begin{tikzpicture}
			\node (X) {$(X, \tau_{X})$};
			\node (Y1) [below of=X] {$(A, \tau_{\mathcal{S}(A)})$};
			\node (Y2) [node distance=3cm, right of=Y1] {$(A, \tau_{\mathcal{S}(A)})$};
			\draw[->] (X) to node {$f_{*}$} (Y2);
			\draw[->] (X) to node [swap] {$f$} (Y1);
			\draw[->] (Y1) to node [swap] {$R$} (Y2);
		  \end{tikzpicture}
	\end{center}
\end{corollary}

\begin{proof}
	Define $R$ as follows.
	For all $x \in X$, let $R(f(x)) = f_{*}(x)$ and for all $B \notin \Ima f$, let $R(B) = B$.
	It is clear that $R(B) \subseteq B$ for all $B$ in $\mathcal{P}(A)$.
	Hence, by proposition~\ref{refinement-map-puf}, $R$ is continuous.
	Consider for every $x$ in $X$, $(R \circ f)(x) = R(f(x)) = f_{*}(x)$, i.e.\ $R$ is the continuous map that makes the above triangle commute.
\end{proof}

It is clear that the map $R$ defined above generalizes the way we obtain a refinement map from a companion map.

\begin{proposition}\label{subfamily-map-puf}
	Let X be a set and $(\mathcal{P}(X), \tau_{\mathcal{S}(X)})$ be the puf topology on $\mathcal{P}(X)$, then any map $S:(\mathcal{P}(X), \tau_{\mathcal{S}(X)}) \rightarrow (\mathcal{P}(D), \tau_{\mathcal{S}(D)})$ defined as $S(A) = A \cap D$ for some subspace $D$ in $X$ is a continuous map.
\end{proposition}

\begin{proof}
	We will show that $S^{-1}(\mathcal{U}_{D}(d)) = \mathcal{U}_{X}(d)$ for all $d$ in $D$ which implies that $S$ is continuous.
	For all $d$ in $D$, consider the following
	\begin{align*}
		\begin{aligned}
			A \in S^{-1}(\mathcal{U}_{D}(d)) & \iff S(A) \in \mathcal{U}_{D}(d) \\
			& \iff d \in S(A) = A \cap D \\
			& \iff d \in A \\
			& \iff A \in \mathcal{U}_{X}(d)
		\end{aligned}
	\end{align*}
	.
	Therefore, the conclusion of the proposition follows.
\end{proof}

\begin{corollary}\label{subfamily-map-coslice}
	Let $(X, \tau)$ be a topological space and $A$ be a set.
	For any continuous map $f : (X, \tau) \rightarrow (A, \tau_{\mathcal{S}(A)})$ and any set $D \subseteq A$ such that the map $g : (X, \tau) \rightarrow (D, \tau_{\mathcal{S}(D)})$ defined as $g(x) = f(x) \cap D$ is continuous, there exists a continuous map \linebreak ${ S : (A, \tau_{\mathcal{S}(A)}) \rightarrow (D, \tau_{\mathcal{S}(D)}) }$ that makes the triangle below commute, i.e.\ $S$ is a morphism from $f$ to $g$ in the coslice category $(X, \tau) \downarrow \text{\textbf{Top}}$.
	\begin{center}
		\begin{tikzpicture}
			\node (X) {$(X, \tau_{X})$};
			\node (A) [below of=X] {$(A, \tau_{\mathcal{S}(A)})$};
			\node (D) [node distance=3cm, right of=A] {$(D, \tau_{\mathcal{S}(D)})$};
			\draw[->] (X) to node {$g$} (D);
			\draw[->] (X) to node [swap] {$f$} (A);
			\draw[->] (A) to node [swap] {$S$} (D);
		  \end{tikzpicture}
	\end{center}
\end{corollary}

\begin{proof}
	Define the map $S$ as follows.
	For each $B$ in $\mathcal{P}(A)$, let $S(B) = B \cap D$.
	As a direct consequence of proposition~\ref{subfamily-map-puf}, $S$ is continuous.
	Moreover, for every $x$ in $X$, ${ (S \circ f)(x) = S(f(x)) = f(x) \cap D = g(x) }$, i.e.\ $S$ is the continuous map that makes the above triangle commute.
\end{proof}

Similarly, the map $S$ defined above generalizes the way we obtain restrictions of open set assignments.


\begin{proposition}\label{continuous-induced-on-puf-topology}
	Let $(X, \tau_X)$ and $(Y, \tau_Y)$ be two topological spaces and $t : X \rightarrow Y$ be a (surjective) function.
	The map $T : (\mathcal{P}(X), \tau_{\mathcal{S}(X)}) \rightarrow (\mathcal{P}(Y), \tau_{\mathcal{S}(Y)})$ defined as
	\[
		T(A) = \{ t(a) : a \in A \} =  t(A)
	\]
	is a (surjective) continuous map.
\end{proposition}

\begin{proof}
	Since any subbasic open set in $\mathcal{P}(Y)$ is of the form $\mathcal{U}_{Y}(y)$ for some $y \in Y$, then to prove that $T$ is continuous it suffices to show that $T^{-1}(\mathcal{U}_{Y}(y))$ is open for all $y \in Y$.
	Let $y \in Y$, we have two possibilities for $y$ as follows.
	If there is no $x \in X$ such that $t(x) = y$, then $T^{-1}(\mathcal{U}_{Y}(y)) = \phi$ since $y \notin t(A)$ for all $A \in \mathcal{P}(X)$.
	If there is some $x \in X$ such that $t(x) = y$, then define $L(y)$ as $\{ x \in X : t(x) = y \}$; hence, $L(y) \neq \phi$.
	Consider the following
	\begin{align*}
		\begin{aligned}
			A \in T^{-1}(\mathcal{U}_{Y}(y)) & \iff T(A) \in \mathcal{U}_{Y}(y) \\
			& \iff y \in T(A) \\
			& \iff y \in t(A) \\
			& \iff \exists x [ x \in L(y) \cap A ] \\
			& \iff \exists x [\{ L(y), A \} \subseteq \mathcal{U}_{X}(x)] \\
			& \iff A \in \bigcup\limits_{x \in L(y)} \mathcal{U}_{X}(x)
		\end{aligned}
	\end{align*}
	Clearly, $ T^{-1}(\mathcal{U}_{Y}(y)) = \bigcup\limits_{x \in L(y)} \mathcal{U}_{X}(x)$ which implies that $T$ is a continuous map.
	If $t$ is surjective, then to prove that $T$ is surjective, let $B \in \mathcal{P}(Y)$.
	Because $t$ is surjective, we can write $B = t(t^{-1}(B))$, and thus $A = t^{-1}(B) \in \mathcal{P}(X)$ is a set satisfying $T(A) = B$ which ends the proof.
\end{proof}

\begin{corollary}
	For any topological space $(X, \tau)$, let $F(X) = (\mathcal{P}(X), \tau_{\mathcal{S}(X)})$ and for any continuous map $t : (X, \tau_X) \rightarrow (Y, \tau_Y)$, let $F(t) = T : F(X) \rightarrow F(Y)$ where $T$ is the map defined as in the statement of proposition~\ref{continuous-induced-on-puf-topology}, then $F$ is an endofunctor on the category \textbf{Top}.
\end{corollary}

\begin{proof}
	For any continuous map $t : (X, \tau_X) \rightarrow (Y, \tau_Y)$, the map $T : F(X) \rightarrow F(Y)$ \linebreak is also continuous as a direct consequence of proposition~\ref{continuous-induced-on-puf-topology}.
	Moreover, let \linebreak ${ t_1 : (X, \tau_X) \rightarrow (Y, \tau_Y) }$ and $t_2 : (Y, \tau_Y) \rightarrow (Z, \tau_Z)$ be any two continuous maps, we want to show that $F(t_2) \circ F(t_1) = F(t_2 \circ t_1)$.
	Let $T_1 = F(t_1)$ and $T_2 = F(t_2)$, then consider the following,
	\begin{align*}
		\begin{aligned}
			(T_2 \circ T_1)(A) & = T_2(T_1(A)) \\
			& = T_2(t_1(A)) \\
			& = t_2(t_1(A)) \\
			& = (t_2 \circ t_1)(A)
		\end{aligned}
	\end{align*}
	.
	Obviously, $F(\text{id}_X) = \text{id}_{F(X)} = \text{id}_{\left( \mathcal{P}(X), \tau_{\mathcal{S}(X)} \right)}$ for all $X$.
	Thus, $F$ is a functor and since it is a functor on \textbf{Top}, then it is an endofunctor.
\end{proof}

\section[Characterization of paracompactness and metacompactness]{Characterization of paracompactness and meta-\linebreak compactness}\label{sec:characterization-of-paracompactness-and-metacompactness}

We introduce characterizations of paracompactness and metacompactness using companion maps.

\begin{lemma}\label{locally-finite-companion-maps}
	Let $(X, \tau)$ be a topological space and $\mathcal{O} = \{ U_a : a \in A \}$ be an open cover of $X$. $\mathcal{O}$ is locally finite iff there exists an open neighborhood assignment $N:X \rightarrow \tau$ such that $\lvert \bigcup f_{\mathcal{O}}(N(x)) \rvert < \omega_0$ for all $x$ in $X$ where $f_{\mathcal{O}} : (X, \tau) \rightarrow (\mathcal{P}(A), \tau_{\mathcal{S}(A)})$ is the companion map of the open covering assignment induced by the open cover $\mathcal{O}$.
\end{lemma}

\begin{proof}
	Assume that $\mathcal{O}$ is locally finite.
	Hence, for each $x \in X$ there is an open neighborhood $V_x$ of $x$ that intersects $U_a$ for finitely many $a$ in $A$.
	Thus, the set \linebreak ${ D_x = \{ a \in A : U_a \cap V_x \neq \phi \} }$ is finite for all $x$ in $X$.
	Define an open neighborhood assignment $N : X \rightarrow \tau$ as $N(x) = V_x$.
	For any $y \in N(x)$, consider \[ f_{\mathcal{O}}(y) = \{ a \in A : y \in U_a \} \subseteq \{ a \in A : U_a \cap N(x) \neq \phi \} = D_x \]; hence, $\bigcup f_{\mathcal{O}}(N(x)) = \bigcup\limits_{y \in N(x)} f_{\mathcal{O}}(y) \subseteq D_x$ which implies that $\lvert \bigcup f_{\mathcal{O}}(N(x)) \rvert \leq |D_x| < \omega_0$.
	This proves the sufficiency part.

	Conversely, assume that there exists an open neighborhood assignment $N : X \rightarrow \tau$ \linebreak such that $\lvert \bigcup f_{\mathcal{O}}(N(x)) \rvert < \omega_0$ for all $x$ in $X$.
	Take any $x$ in $X$.
	Consider the set \linebreak ${ D_x = \{ a \in A : N(x) \cap U_a \neq \phi \} }$ and take any $a \in D_x$, then $N(x) \cap U_a \neq \phi$.
	Let $y \in N(x) \cap U_a$, i.e.\ $y \in N(x)$ and $y \in U_a$; hence, $a \in f_{\mathcal{O}}(y) \subseteq \bigcup f_{\mathcal{O}}(N(x))$.
	Since $a$ was taken arbitrarily, then $D_x \subseteq \bigcup f_{\mathcal{O}}(N(x))$ which implies that $|D_x| \leq \lvert \bigcup f_{\mathcal{O}}(N(x)) \rvert < \omega_0$.
	Therefore, $N(x)$ is an open neighborhood of $x$ that intersects $U_a$ for finitely many $a$ in $A$.
	Since $x$ was taken arbitrarily too, then it is clear that $\mathcal{O}$ is locally finite.
	This proves the necessary part.
\end{proof}

\begin{theorem}[\textbf{characterization of paracompactness}]\label{paracompactness-by-companion-maps}
	A topological space $(X, \tau)$ is paracompact iff for every open covering assignment $C : A \rightarrow \tau$, there is another open covering assignment $C_r : D \rightarrow \tau$ such that for all $d$ in $D$, $C_r(d) \subseteq C(a)$ for some $a$ in $A$, and there exists an open neighborhood assignment $N : X \rightarrow \tau$ such that $\lvert \bigcup f_{C_r}(N(x)) \rvert < \omega_0$ for all $x \in X$ where $f_{C_r}$ is the companion map of $C_r$.
\end{theorem}

\begin{proof}
	Assume that $X$ is a paracompact space and let $C : A \rightarrow \tau$ be an open covering map.
	Therefore, $\{ C(a) : a \in A \}$ is an open cover, and thus it has an open refinement $\{ U_d : d \in D \}$ that is locally finite for some set $D$.
	Hence, $C_r : D \rightarrow \tau$ defined as $C_r(d) = U_d$ for all $d$ in $D$ is an open covering assignment such that $C_r(d) \subseteq C(a)$ for some $a$ in $A$.
	Applying lemma~\ref{locally-finite-companion-maps}, we obtain an open neighborhood assignment $N : X \rightarrow \tau$ satisfying the conclusion where $f_{C_r}$ is the companion map of the open covering assignment $C_r$.

	Conversely, assume that the necessary condition holds.
	Let $\mathcal{O} = \{ U_a : a \in A \}$ be any open cover of $X$, and thus the assignment $C : A \rightarrow \tau$ defined as $C(a) = U_a$ is an open covering assignment.
	Therefore, there is another open covering map $C_r : D \rightarrow \tau$ such that for all $d$ in $D$, $C_r(d) \subseteq C(a)$ for some $a$ in $A$, and there exists an open neighborhood assignment $N : X \rightarrow \tau$ such that $\lvert \bigcup f_{C_r}(N(x)) \rvert < \omega_0$ for all $x \in X$.
	Hence, $\{ C_r(d) : d \in D \}$ is an open refinement of $\mathcal{O}$.
	Thus, by lemma~\ref{locally-finite-companion-maps}, again, $\{ C_r(d) : d \in D \}$ is locally finite which ends the proof.
\end{proof}


\begin{theorem}[\textbf{characterization of metacompactness}]\label{metacompactness-by-companion-maps}
	A topological space $(X, \tau)$ is metacompact iff for every for every open covering assignment $C : A \rightarrow \tau$, there is another open covering assignment $C_r : D \rightarrow \tau$ such that for all $d$ in $D$, $C_r(d) \subseteq C(a)$ for some $a$ in $A$ and that satisfies $\lvert f_{C_r}(x) \rvert < \omega_0$ for all $x \in X$ where $f_{C_r}$ is the companion map of $C_r$.
\end{theorem}

\begin{proof}
	Assume that $X$ is a metacompact space and let $C : A \rightarrow \tau$ be an open covering map.
	Therefore, $\{ C(a) : a \in A \}$ is an open cover, and thus it has an open refinement $\{ U_d : d \in D \}$ that is point finite for some set $D \subseteq X$.
	Hence, $C_r : D \rightarrow \tau$ defined as $C_r(d) = U_d$ for all $d$ in $D$ is an open covering assignment such that $C_r(d) \subseteq C(a)$ for some $a$ in $A$.
	From the fact that $\{ U_d \}_{ d \in D}$ is point finite, then for all $x \in X$, $x$ is an element of some open sets $\{ U_{d_1(x)}, \dots, U_{d_{n_x}(x)} \}$ in $\{ U_d \}_{ d \in D}$.
	Therefore, $f_{C_r}(x) = \{ d_1(x), \dots, d_{n_x}(x) \}$ from which the conclusion follows.

	Conversely, assume that the necessary condition holds.
	Let $\mathcal{O} = \{ U_a : a \in A \}$ be any open cover of $X$, and thus the assignment $C : A \rightarrow \tau$ defined as $C(a) = U_a$ is an open covering assignment.
	Hence, there is another open covering map $C_r : D \rightarrow \tau$ such that for all $d$ in $D$, $C_r(d) \subseteq C(a)$ for some $a$ in $A$ that satisfies $\lvert f_{C_r}(x) \rvert < \omega_0$ for all $x \in X$.
	In other words, $\{ C_r(d) : d \in D \}$ is an open refinement of $\mathcal{O}$ such that $x \in C_r(d)$ for finitely many $d$ in $D$.
	Therefore, $\{ C_r(d) : d \in D \}$ is point finite which completes the proof.
\end{proof}

Observe that the companion map $f_{C_r}$ in both theorem~\ref{paracompactness-by-companion-maps} and theorem~\ref{metacompactness-by-companion-maps} is naturally related to the concepts of paracompactness and metacompactness since the concepts of ``locally finite'' and ``point finite'' are related to how sets from some family contain points of the space which is what companion maps are about.

\clearpage
\chapter{Applications to the theory of D-spaces}\label{ch:applications-to-the-theory-of-d-spaces}
\section{Preparation}\label{sec:preparation}

To prove our main theorems, we will first need a body of lemmas.

\begin{lemma}\label{refinement-in-d-space}
    A topological space $(X, \tau)$ is a D-space iff every open neighborhood assignment $N : X \rightarrow \tau$ has a neighborhood refinement $N^{*} : X \rightarrow \tau$ with a closed discrete kernel.
\end{lemma}

\begin{proof}
    Assume that $(X, \tau)$ is a D-space, and set $N^{*} = N$, then the conclusion follows trivially.
    Conversely, let $N$ be an open neighborhood assignment and assume that it has a neighborhood refinement $N^{*}$ such that for some closed discrete set $D \subseteq X$, $\bigcup N^{*}(D)= X$.
    Consider $X = \bigcup\limits_{d \in D} N^{*}(d) \subseteq \bigcup\limits_{d \in D} N(d)$.
    Hence, the conclusion follows.
\end{proof}

\begin{lemma}\label{closed-discrete-subset}
    Let $(X, \tau)$ be a topological space and $D \subseteq X$ be a closed discrete set, then any set $F \subseteq D$ is also closed and discrete.
\end{lemma}

\begin{proof}
    Let $F \subseteq D$ be any set, then we will first show that $F$ is discrete.
    Since $x \in F$ implies that $x \in D$, then there exists an open set $U \in \tau$ containing $x$ such that $U \cap D = \{ x \}$ which implies that $U \cap F = \{ x \}$.
    To prove that $F$ is closed, take any $x \in X \setminus F$.
    Either $x \in D$ or $x \in X \setminus D$.
    If $x \in D$, then there exists an open set $U \in \tau$ containing $x$ such that $U \cap D = \{ x \}$.
    Since $F \subseteq D$ and $x \notin F$, then $U \cap F = \phi$; hence, $x \in U \subseteq X \setminus F$.
    On the other hand, if $x \in X \setminus D$, then it is clear that there exists an open set $U \in \tau$ such that $x \in U \subseteq X \setminus D \subseteq X \setminus F$.
    Either way, we conclude that $F$ is closed.
\end{proof}


\begin{lemma}\label{closed-discrete-refinement}
	Let $(X, \tau)$ be a topological space and $N : X \rightarrow \tau$ be an open neighborhood assignment with a closed discrete kernel $D \subseteq X$.
	Define a function $N^{*} : X \rightarrow \mathcal{P}(X)$ as follows. $N^{*}(x) = N(x)$ for all $x \in X \setminus D$ and $N^{*}(x) = N(x) \setminus (D \setminus \{ x \})$ for all $x \in D$.
	The function $N^{*}$ is a neighborhood refinement of $N$ and $D$ is a kernel of $N^{*}$.
\end{lemma}

\begin{proof}
	Since for any $d \in D$, $D \setminus \{ d \} \subseteq D$, then by lemma~\ref{closed-discrete-subset}, $D \setminus \{ d \}$ is also \linebreak  closed and discrete;
	hence, the set $N(d) \setminus (D \setminus \{ d \})$ is open.
	Moreover, \linebreak ${ N(d) \setminus (D \setminus \{ d \}) = (N(d) \setminus D) \cup \{ d \} }$ for all $d \in D$.
	Therefore, $d \in N(d)$ for all $d \in D$.
	If $x \in X \setminus D$, then $N^{*}(x) = N(x)$ which is an open neighborhood of $x$.
	Hence, $N^{*}$ an open neighborhood assignment.
	It is clear that $N^{*}(x) \subseteq N(x)$ for all $x \in X$, thus, $N^{*}$ is a neighborhood refinement of $N$.
	It remains to show that $D$ is also a kernel of $N^{*}$.
	Take any $x \in X$, then either $x \in X \setminus D$ or $x \in D$.
	If $x \in X \setminus D$, then since $\bigcup N(D) = X$, we conclude that $x \in N(d)$ for some $d \in D$.
	Since $x \notin D$ and $N^{*}(d) = N(d) \setminus (D \setminus \{ d \})$, then obviously $x \in N^{*}(d)$.
	If $x = d \in D$, then $x \in N(d)$; hence, $x \in N(d) \setminus (D \setminus \{ d \}) = N^{*}(d)$.
	Therefore, $\bigcup N^{*}(D) = X$.
\end{proof}

\begin{lemma}\label{companion-maps-d-spaces}
	A topological space $(X, \tau)$ is a D-space iff for any continuous map \linebreak ${ f : (X, \tau) \rightarrow \left(\mathcal{P}(X), \tau_{\mathcal{S}(X)}\right) }$ such that $f(x) \in \mathcal{U}_{X}(x)$ for all $x \in X$, there exists a continuous neighborhood refinement map $f_{*}$ of $f$ and a closed discrete set $D \subseteq X$ such that the map $g : (X, \tau) \rightarrow (\mathcal{P}(D), \tau_{\mathcal{S}(D)})$ defined as $g(x) = f_{*}(x) \cap D$ is continuous and satisfies $g^{-1}\left(\bigcup \mathcal{U}(D)\right) = X \iff g^{-1}(\{\phi\}) = \phi$.
\end{lemma}

\begin{proof}
	Assume that $(X, \tau)$ is a D-space, then let $f : (X, \tau) \rightarrow (\mathcal{P}(X), \tau_{\mathcal{S}(X)})$ be any continuous map such that $f(x) \in \mathcal{U}_{X}(x)$, i.e.\ $x \in f(x)$.
	Define $N : X \rightarrow \tau$ as ${ N(x) = f^{-1}(\mathcal{U}_{X}(x)) }$ for all $x \in X$.
	Clearly, $N$ is an open neighborhood assignment, and thus it has a closed discrete kernel $D \subseteq X$.
	Moreover, let $f_{*} = f$.
	By corollary~\ref{subcover-map}, the map ${ g : (X, \tau) \rightarrow (\mathcal{P}(D), \tau_{\mathcal{S}(D)}) }$ defined as $g(x) = f(x) \cap D = f_{*} \cap D$ is continuous and satisfies $g^{-1}(\{\phi\}) = \phi$.
	Conversely, assume that the necessary condition holds.
	We will show that $X$ is a D-space.
	Let $N : X \rightarrow \tau$ be any open neighborhood assignment and let $f$ be the companion map of $N$.
	Clearly, $f$ is continuous and $f(x) \in \mathcal{U}_{X}(x)$ for all $x \in X$.
	From the hypothesis, there exists a continuous neighborhood refinement map $f_{*}$ and a closed discrete set $D \subseteq X$ such that the function $g$ defined as in the hypothesis is continuous such that $g^{-1}(\{\phi\}) = \phi$.
	Using corollary~\ref{subcover-map}, again, we conclude that $D$ is a kernel for the open neighborhood assignment $N^{*}:X \rightarrow \tau$ where $N^{*}(x) = f_{*}^{-1}(\mathcal{U}_{X}(x))$.
	Hence, from lemma~\ref{refinement-in-d-space}, $X$ is a D-space.
\end{proof}

Lemma~\ref{companion-maps-d-spaces} implies that we can use the companion maps as an alternative to open neighborhood assignments when discussing D-spaces.

\begin{lemma}\label{forcing-point-in-kernel}
	Let $(X, \tau)$ be a topological space and $d_0$ be a point in $X$.
	If ${ N : X \rightarrow \tau }$ is an open neighborhood assignment such that only open set containing $d_0$ among ${ \{ N(x) : x \in X \} }$ is $N(d_0)$, then $d_0 \in D$ for any kernel $D$ of $N$.
	Moreover, for any neighborhood refinement $N^{*} : X \rightarrow \tau$ of $N$, it also holds that $d_0 \in D$ for any kernel $D$ of $N^{*}$.
\end{lemma}

\begin{proof}
	For $N$, assume that there is a kernel $D$ such that $d_0 \notin D$, then from the fact that $D$ is a kernel and $d_0 \in X$, there exists some $d$ in $D$ such that $d_0 \in N(d)$.
	Clearly, $d \neq d_0$ from our assumption that $d_0 \notin D$.
	This contradicts the fact that the only open set containing $d_0$ among $\{ N(x) : x \in X \}$ is $N(d_0)$.
	Moreover, let $N^{*}: X \rightarrow \tau$ be any neighborhood refinement of $N$, i.e.\ $x \in N^{*}(x) \subseteq N(x)$ for all $x$ in $X$.
	Assume that $D^{*}$ is a kernel of $N^{*}$ such that $d_0 \notin D^{*}$.
	Clearly, $X = \bigcup N^{*}(D^{*}) \subseteq N(D^{*})$, and hence $D^{*}$ is a kernel of of $N$ too.
	Similarly, since $d_0 \in X$ and $D^{*}$ is a kernel of $N^{*}$, then there exists some $d$ in $D^{*}$ such that $d_0 \in N^{*}(d)$ which implies that $d_0 \in N(d)$.
	However, $d_0 \neq d$ since $d_0 \notin D^{*}$.
	This also contradicts the fact that the only open set containing $d_0$ among $\{ N(x) : x \in X \}$ is $N(d)$.
\end{proof}

\begin{lemma}\label{closed-discrete-kernel}
	Let $(X, \tau)$ be a $T_1$ topological space and $N : X \rightarrow \tau$ be an open neighborhood assignment.
	If a set $D \subseteq X$ satisfies
	\begin{enumerate}
		\item $\bigcup N(D) = \bigcup\limits_{d \in D} N(d) = X$, i.e.\ $D$ is a kernel of $N$
		\item for all $d$ in $D$, it holds that $d \notin N(x)$ for all $x \in D \setminus \{ d \}$
	\end{enumerate}
	, then $D$ is closed and discrete.
\end{lemma}

\begin{proof}
	Take any $d$ in $D$.
	For any $d\,'$ in $D$ distinct from $d$, we have $d\,' \notin N(d)$ since ${ d \in D \setminus \{d\,'\} }$.
	Hence, $(D \setminus \{d\}) \cap N(d) = \phi$; however, $d \in N(d)$, and thus ${ N(d) \cap D = \{ d \} }$ which proves that $D$ is discrete.
	To prove that $D$ is closed, let $x \in X \setminus D$.
	Since ${ x \in X = \bigcup N(D) }$, then $x \in N(d)$ for some $d$ in $D$.
	Moreover, since $X$ is $T_1$, $U = N(d) \setminus \{ d \}$ is an open set containing $x$.
	Clearly, $U \cap D = \phi$, and hence $x \in U \subseteq X \setminus D$ which ends the proof.
\end{proof}

\section{A characterization of D-spaces}\label{sec:a-characterization-of-d-spaces}

In the following theorem, we show that we can express the property D solely through continuous maps.

\begin{theorem}\label{main-theorem-d-space}
	A $T_1$ topological space $(X, \tau)$ is a D-space iff for any continuous map $f : (X, \tau) \rightarrow (\mathcal{P}(X), \tau_{\mathcal{S}(X)})$ such that $f(x) \in \mathcal{U}_{X}(x)$ for all $x \in X$, there is a continuous neighborhood refinement map $f_{*} : (X, \tau) \rightarrow (\mathcal{P}(X), \tau_{\mathcal{S}(X)})$ of $f$ and a set $D \subseteq X$ such that:

	\begin{enumerate}
		\item\label{itm:main-theorem-d-space-1} The map $\mathcal{I} : ( D, \tau_D) \rightarrow(\mathcal{P}(D), \tau_{\mathcal{S}(D)})$ from the subspace topology on $D$ to the puf topology induced by $D$, defined as
        \[
            \mathcal{I}( d ) = \{ d \}
        \]
        is an injective continuous map (monomorphism).
        \item\label{itm:main-theorem-d-space-2} The map $g : (X, \tau) \rightarrow (\mathcal{P}(D), \tau_{\mathcal{S}(D)})$ defined as $g(x) = f_{*}(x) \cap D$ is a continuous map such that $g^{-1}(\{ \phi \}) = \phi$.
        \item\label{itm:main-theorem-d-space-3} The set $\Delta(D) = D^2$ is a subspace of the pullback $P$ of the two continuous functions $g$ and $\mathcal{I}$.
        Hence, $g(d) = \mathcal{I}(d) = \{d\}$ for all $d$ in $D$.
	\end{enumerate}
\end{theorem}

The following commutative diagram illustrates theorem~\ref{main-theorem-d-space}.

\begin{center}
    \begin{tikzpicture}
        \node (P)  {$P$};
        \node (D)  [right of=P] {$D$};
        \node (PD) [below of=D] {$\mathcal{P}(D)$};
        \node (X)  [below of=P] {$X$};
        \draw[->] (P) to node {$p_2$} (D);
        \draw[->] (D) to node {$\mathcal{I}$} (PD);
        \draw[->] (P) to node [swap] {$p_1$} (X);
        \draw[->] (X) to node [swap] {$g$} (PD);
    \end{tikzpicture}
    \[
        P \supseteq \Delta(D)
    \]
\end{center}

\begin{proof}
	Assume that $X$ is a D-space.
	Let $f : (X, \tau) \rightarrow (\mathcal{P}(X), \tau_{\mathcal{S}(X)})$ be a continuous map such that $f(x) \in \mathcal{U}_{X}(x)$ for all $x \in X$.
	Define the assignment ${ N : X \rightarrow \tau }$ as ${ N(x) = f^{-1}(\mathcal{U}_{X}(x)) }$ for all $x \in X$, then $N$ is clearly an open neighborhood assignment.
	Therefore, it has a closed discrete kernel $D \subseteq X$.
	Applying lemma~\ref{closed-discrete-refinement}, we obtain a neighborhood refinement $N^{*}$ of $N$ such that $N^{*}(x) = N(x)$ for all $x \in X \setminus D$ and \linebreak ${ N^{*}(x) = N(x) \setminus (D \setminus \{ x \}) }$ for all $x \in D$.
	Let $f_{*}$ be the companion map of the open neighborhood assignment $N^{*}$.
	Applying proposition~\ref{refinement-companion-map}, we conclude that $f_{*}$ is a neighborhood refinement map of $f$.
	First, we will prove~\ref{itm:main-theorem-d-space-1}.

	Since $D$ is a discrete set, then the subspace topology of $D$ is the discrete topology on $D$.
	Hence, $\{ d \}$ is open for all $d \in D$.
	To prove that $\mathcal{I}$ is injective, let $\mathcal{I}(d_1) = \mathcal{I}(d_2)$ for $d_1, d_2 \in D$.
	Thus, $\{ d_1 \} = \{ d_2 \}$ which implies that $d_1 = d_2$.
	It is obvious that $\mathcal{I}$ is continuous since any subset of $D$ is open in the subspace topology on $D$.
	Second, by applying corollary~\ref{subcover-map} on $f_{*}$ and $D$, then~\ref{itm:main-theorem-d-space-2} also holds.

	Finally, we will prove~\ref{itm:main-theorem-d-space-3}.
	Since the category $\textbf{Top}$ has all finite products and equalizers, then the pullback of $g$ and $\mathcal{I}$ definitely exists (see Awodey, 2010)$^{\text{\cite{awodey}}}$.
	We will proceed by calculating this pullback.
	First, we take the product topology on $X \times D$ of $(X, \tau)$ and $(D, \tau_{D})$.
	Let $\pi_1 : X \times D \rightarrow X$ and $\pi_2 : X \times D \rightarrow D$ be the two projection maps of the product topology.
	Consider the equalizer $E$ of $g \circ \pi_1$ and $\mathcal{I} \circ \pi_2$.

    \begin{center}
        \begin{tikzpicture}
            \node (E)  {$E$};
            \node (XD) [node distance=2cm, right of=E] {$X \times D$};
            \node (PD) [node distance=3cm, right of=XD] {$\mathcal{P}(D)$};
            \draw[->] (E) to node {$\iota$} (XD);
            \draw[->] (XD.8) to node {$g \circ \pi_1$} (PD.174);
            \draw[->] (XD.351) to node [swap] {$\mathcal{I} \circ \pi_2$} (PD.190);
        \end{tikzpicture}
    \end{center}

	We claim that $\Delta(D) = D^{2}$ is a subspace of $E$, i.e.\ $\Delta(D) \subseteq E$.
	Since for all $d \in D$, $\mathcal{I} \circ \pi_2(d, d) = \mathcal{I}(d) = \{ d \}$, we must prove that $g \circ \pi_2(d, d) = \{ d \}$.
	By the definition of $g$, $g(d) = f_{*}(d) \cap D$.
	Since $d \in f_{*}(d)$, then $\{ d \} \subseteq f_{*}(d) \cap D$.
	To show that $\{ d \} = f_{*}(d) \cap D$, assume that there exists $d\,' \in f_{*}(d) \cap D$ such that $d\,' \neq d$.
	Hence, $d\,' \in D$ and $d\,' \in f_{*}(d)$; thus, $f_{*}(d) \in \mathcal{U}_{X}(d\,')$, i.e.\ $d \in f^{-1}_{*}(\mathcal{U}_{X}(d\,'))$.
	However, \[ f^{-1}_{*}(\mathcal{U}_{D}(d\,')) = N^{*}(d\,') = N(d\,') \setminus (D \setminus \{ d\,' \}) \] which implies that $d \in N(d\,') \setminus (D \setminus \{ d\,' \})$.
	Consider the following
	\begin{align*}
    	\begin{aligned}
			d \in D \cap \left( N(d\,') \setminus (D \setminus \{ d\,' \}) \right) & \Longrightarrow d \in D \cap \left( (N(d\,') \setminus D) \cup \{ d\,' \} \right) \\
			& \Longrightarrow d \in \left( D \cap (N(d\,') \setminus D) \right) \cup \left( D \cap \{ d\,' \} \right) \\
			& \Longrightarrow d \in \phi \cup \{d\,'\} = \{d\,'\}
      	\end{aligned}
	\end{align*}

	As a consequence, $d = d\,'$ which contradicts the assumption.
	Therefore, we have ${ g(d) = \{ d \} = \mathcal{I}(d) }$ which implies that $\Delta(D)$ is indeed a subspace of $E$.

	Conversely, assume that the necessary condition holds.
	We will show that $X$ is a D-space.
	Let $f:(X, \tau) \rightarrow (\mathcal{P}(X), \tau_{\mathcal{S}(X)})$ be any continuous map such that $f(x) \in \mathcal{U}_{X}(x)$, then there is a neighborhood refinement $f_{*}$ of $f$ and a set $D \subseteq X$ such that~\ref{itm:main-theorem-d-space-1},~\ref{itm:main-theorem-d-space-2}, and~\ref{itm:main-theorem-d-space-3} holds.
	It suffices to show that $D$ is closed and discrete, then by applying lemma~\ref{companion-maps-d-spaces} on $f_{*}$ and $D$, we can infer that $X$ is a D-space.

    Again, the pullback $P$ in~\ref{itm:main-theorem-d-space-3} is just the equalizer of $g \circ \pi_1$ and $\mathcal{I} \circ \pi_2$, where $\pi_1$ and $\pi_2$ are the projection functions of the product space $X \times D$.
    From~\ref{itm:main-theorem-d-space-3}, $\Delta(D) \subseteq P$.
    Let $d \in D$, then $(d, d) \in \Delta(D)$.
    Hence, $g(\pi_1(d, d)) = \mathcal{I}(\pi_2(d, d))$, i.e.\ $g(d) = \mathcal{I}(d) = \{ d \}$ which implies that $d \in g^{-1}(\mathcal{U}_{D}(d))$.
    We claim that if $d\,' \neq d$ in $D$, then $d\,' \notin g^{-1}(\mathcal{U}_{D}(d))$.
    To prove this claim assume not, i.e.\ $d\,' \in g^{-1}(\mathcal{U}_{D}(d))$ for some $d\,' \neq d$ in $D$.
    Hence, ${ g(d\,') \in \mathcal{U}_{D}(d) }$; however, since $(d\,', d\,') \in \Delta(D)$, then $g(d\,') = \mathcal{I}(d\,') = \{ d\,' \}$.
    This means that $d \in g(d\,') = \{ d\,' \}$, and thus $d = d\,'$, a contradiction.
    Therefore, for any $d \in D$, $g^{-1}(\mathcal{U}_{D}(d))$ is a neighborhood of $d$ which satisfies $g^{-1}(\mathcal{U}_{D}(d)) \cap D = \{ d \}$.
    Define an open neighborhood assignment as $N^{*}(x) = f^{-1}_{*}(\mathcal{U}_{X}(x))$ for all $x \in X$, then from~\ref{itm:main-theorem-d-space-2} and by corollary~\ref{subcover-map}, $D$ is a kernel for $N^{*}$.
    Moreover, $N^{*}(d) = g^{-1}(\mathcal{U}_{D}(d))$ for all $d \in D$.
    Hence, $N^{*}(d) \cap D = \{ d \}$ for all $d \in D$.
    By lemma~\ref{closed-discrete-kernel}, the desired conclusion follows.

\end{proof}

\begin{corollary}\label{main-theorem-d-space-corollary}
	A $T_1$ topological space $(X, \tau)$ is a D-space iff for any continuous map $f : (X, \tau) \rightarrow (\mathcal{P}(X), \tau_{\mathcal{S}(X)})$ such that $f(x) \in \mathcal{U}_{X}(x)$ for all $x \in X$, there is a continuous neighborhood refinement map $f_{*} : (X, \tau) \rightarrow (\mathcal{P}(X), \tau_{\mathcal{S}(X)})$ of $f$ and a set $D \subseteq X$ such that the map $g:(X, \tau) \rightarrow (\mathcal{P}(D), \tau_{\mathcal{S}(D)})$ defined as $g(x) = f_{*}(x) \cap D$ is continuous and satisfies $g^{-1}(\{\phi\}) = \phi$ together with $g(d) = \{d\}$ for all $d$ in $D$.
\end{corollary}

\begin{proof}
	The sufficiency part is a direct consequence of theorem~\ref{main-theorem-d-space}.
	Conversely, assume that the necessary condition holds.
	We only need to show that the set $D$ is closed and discrete, then by applying lemma~\ref{companion-maps-d-spaces}, we obtain that $X$ is a D-space.
	To prove this fact, let $N^{*}$ be the associated open neighborhood assignment of $f_{*}$.
	Since ${ g^{-1}(\{ \phi \}) = \phi }$, then from corollary~\ref{subcover-map} the assignment $N^{*}\restriction_{D}$ is an open covering assignment, i.e.\ ${ \bigcup N^{*}(D) = X }$.
	Take any $d \in D$, then $g(d) = \{ d \}$ implies that \[ g(d) = f_{*}(d) \cap D = \{ a \in X : d \in N^{*}(a) \} \cap D = \{ a \in D : d \in N^{*}(a) \} = \{ d \} \] which in turn implies that $D \cap N^{*}(d) = \{d\}$.
	Hence, by lemma~\ref{closed-discrete-kernel}, the set $D$ is closed and discrete which ends the proof.
\end{proof}

\begin{remark}\label{main-theorem-d-space-corollary-remark}
	The necessary condition of corollary~\ref{main-theorem-d-space-corollary} implies that the set $D \subseteq X$ is closed and discrete and that it is also a kernel of the neighborhood assignment $N^{*}$ which is associated with the companion map $f_{*}$.
	Moreover, $g(d) = \{ d \}$ is equivalent to $\{ a \in D : d \in N^{*}(a) \} = \{ d \}$ which is equivalent to saying that the for $d$ in $D$, the only open set containing $d$ among $\{ N^{*}(a) : a \in D \}$ is $N^{*}(d)$.
\end{remark}

After looking at theorem~\ref{main-theorem-d-space}, we can see that the essential information related to the property D exists in a subcategory of the coslice category $(X, \tau) \downarrow \text{\textbf{Top}}$.
This subcategory contains all the continuous maps $f:(X, \tau) \rightarrow (\mathcal{P}(X), \tau_{\mathcal{S}(X)})$ such that $f(x) \in \mathcal{U}_{X}(x)$ and all the continuous maps $g:(X, \tau) \rightarrow (\mathcal{P}(D), \tau_{\mathcal{S}(D)})$, for any subset $D$ of $X$, defined as $g(x) = f(x) \cap D$ with the property $g^{-1}(\{ \phi \}) = \phi$.
The morphisms between these maps were discussed in corollaries~\ref{refinement-map-coslice} and~\ref{subfamily-map-coslice}.

\section{On Lindel\"of D-spaces}\label{sec:on-lindelof-d-spaces}

We discussed in section~\ref{sec:a-discussion-on-two-open-problems} Arhangel'skii's approach to question~\ref{qstn:is-lindelof-d}.
One of the reasons we chose to develop the theory of companion maps, is to see whether we can prove that an image of a Lindel\"of D-space is a D-space.
Our approach to carry out a proof of this statement could be described as follows.
Given any regular Lindel\"of D-space $(X, \tau_X)$, a space $(Y, \tau_Y)$, and a surjective continuous map (epimorphism in \textbf{Top}) $t : X \rightarrow Y$, we want to show that $(Y, \tau_Y)$ is a Lindel\"of D-space.
First, since the Lindel\"of property is invariant under continuous surjections, then $(Y, \tau_Y)$ is Lindel\"of.
The essential part is to prove that $(Y, \tau_Y)$ is a D-space.

Given any continuous map $f_Y : (Y, \tau_Y) \rightarrow (\mathcal{P}(Y), \tau_{\mathcal{S}(Y)})$, we want to find a set ${ D_Y \subseteq Y }$ and a continuous refinement map $f_{Y}^{*}$ of $f_Y$ that satisfies the necessary condition of theorem~\ref{main-theorem-d-space}.
To do this we want to use the fact that $(X, \tau_X)$ is a D-space.
In other words, we want to find for each continuous map $f_Y$, a continuous map $f_X : (X, \tau_X) \rightarrow (\mathcal{P}(X), \tau_{\mathcal{S}(X)})$, and in turn, a continuous refinement map $f_X^{*}$ of $f_X$ and a set $D_X \subseteq X$ such that both $f_X^{*}$ and $D_X$ satisfy the necessary condition of theorem~\ref{main-theorem-d-space}, from where we can construct the required $f_{Y}^{*}$ and $D_Y$ through the properties of the continuous maps and using the fact that both $(X, \tau_X)$ and $(Y, \tau_Y)$ are Lindel\"of.
Unfortunately, we didn't succeed in our attempts.
To be more specific, we saw the possibility that either more assumptions should be made or that the puf topology should have better properties (which will be discussed in the later chapter).
As a side note here, sometimes, it is useful to use the following proposition regarding Lindel\"of D-spaces.

\begin{proposition}\label{closed-discrete-in-lindelof-is-countable}
	Let $(X, \tau)$ be a Lindel\"of space, then any closed discrete subset $D$ of $X$ is countable.
\end{proposition}

\begin{proof}
	Let $D$ be a closed discrete subset of $X$.
	Since $D$ is discrete, then for each $d$ in $D$, we can find an open neighborhood $U_d$ such that $U_d \cap D = \{d\}$.
	Since $D$ is closed, then $X \setminus D$ is an open set.
	Hence, the family $\mathcal{O} = \{ U_d : d \in D \} \cup \{ X \setminus D \}$ is an open cover of $X$.
	From the Lindel\"of property, $\mathcal{O}$ reduces to a countable subcover $\mathcal{O}^{*}$.
	Observe that $U_d \in \mathcal{O}^{*}$ for all $d$ in $D$.
	To prove this, assume not, i.e.\ that $U_{d_0} \notin \mathcal{O}^{*}$ for some $d_0$ in $D$.
	Since $\mathcal{O}^{*}$ is an open cover, then there must be an open neighborhood $V_{d_0} \in \mathcal{O}^{*}$ of $d_0$.
	Thus, $V_{d_0} \in \mathcal{O}$ which implies that $V_{d_0} = X \setminus D$ or $V_{d_0} \in  \{ U_d : d \in D \}$.
	The first possibility $V_{d_0} = X \setminus D$ leads to a contradiction since $d_0 \in V_{d_0} = X \setminus D$.
	If $V_{d_0} \in  \{ U_d : d \in D \}$, then $V_{d_0} = U_{d_1}$ for some $d_1 \neq d_0$ or otherwise it will contradict the assumption that $U_{d_0} \notin \mathcal{O}^{*}$.
	This implies that $d_0 \in U_{d_1}$, but $d_0 \in U_{d_1} \cap D = \{d_1\}$ which means that $d_0 = d_1$, a contradiction.
	Therefore, there is an injection from $D$ into $\mathcal{O}^{*}$, i.e.\ $D$ is countable.
\end{proof}

As a direct consequence of proposition~\ref{closed-discrete-in-lindelof-is-countable}, a characterization could be obtained for Lindel\"of D-spaces by adjusting the statement of theorem~\ref{main-theorem-d-space} so that the set $D$ is always countable.

Now we will discuss some general guidelines for any attempt to construct a counterexample.
First, we introduce a new separation property called $\kappa$-exclusiveness in the following definition, then we will see that Lindel\"of spaces with this separation property under certain $\kappa$ are D-spaces.

\begin{definition}
	Let $(X, \tau)$ be a topological space, then $X$ is said to be a $\kappa$-exclusive space if for any $x$ in $X$ and any open neighborhood $U$ of $x$, it holds that for any $A \subset U$ satisfying $x \notin A$ and $|A| \leq \kappa$, there exists an open neighborhood $V$ of $x$ such that $V \subseteq U \setminus A$.
\end{definition}

\begin{theorem}\label{set-in-k-exclusive-is-closed}
	A topological space $(X, \tau)$ is $\kappa$-exclusive space iff any set $A \subseteq X$ of cardinality at most $\kappa$ is closed.
\end{theorem}

\begin{proof}
	Assume that $(X, \tau)$ is $\kappa$-exclusive, and let $A$ be a subset of $X$ such that $\lvert A \rvert \leq \kappa$.
	If $A = \phi$ or $A = X$, then the conclusion is trivial.
	Otherwise, assume that $\phi \neq A \neq X$.
	Since $X$ is an open set, then for each $x \in X \setminus A$, $X$ is an open neighborhood of $x$.
	Thus, by $\kappa$-exclusiveness, for each $x \in X \setminus A$, there exists an open neighborhood $V_x$ of $x$ such that $V_x \subseteq X \setminus A$.
	Therefore, $\bigcup \{ V_x : x \in X \setminus A \} = X \setminus A$ which implies that $X \setminus A$ is an open set in $X$.
	Hence, $A$ is closed.
	The converse is easy to show.
\end{proof}

\begin{corollary}\label{exclusive-t1}
	A $\kappa$-exclusive topological space $(X, \tau)$ with $\kappa \geq 1$ is a $T_1$ space.
\end{corollary}

\begin{proof}
	Since for any $x$ in $X$, the cardinality of $\{ x \}$ is 1, then from theorem~\ref{set-in-k-exclusive-is-closed}, every singleton is closed.
\end{proof}

\begin{theorem}\label{lindelof-k-exclusive-is-d}
	Let $(X, \tau)$ be a Lindel\"of $\kappa$-exclusive space with $\kappa \geq \omega_0$, then $X$ is a D-space.
\end{theorem}

\begin{proof}
	Let $N : X \rightarrow \tau$ be any open neighborhood assignment, then $\{ N(x) : x \in X \}$ is an open cover of $X$.
	Hence, from the Lindel\"of property of $X$, it reduces to a countable subcover $\{ N(d) : d \in D \}$ where $D$ is countable, i.e.\ $\lvert D \rvert = \omega_0$.
	Since $X$ is $\kappa$-exclusive with $\kappa \geq \omega_0$, then $D$ is closed.
	Define a neighborhood refinement $N^{*}$ of $N$ as follows.
	Let $N^{*}(x) = N(x)$ for all $x \notin D$ and $N^{*}(d) = N(d) \setminus (D \setminus \{ d\})$.
	It is clear that $N^{*}(d) \cap D$ = $\{ d \}$ for all $d \in D$.
	Thus, $D$ is discrete.
	Therefore, by lemma~\ref{refinement-in-d-space}, $X$ is a D-space.
\end{proof}

We can conclude, from theorem~\ref{lindelof-k-exclusive-is-d}, that if some Lindel\"of space is not a D-space, then it is not a $\kappa$-exclusive space with $\kappa \geq \omega_0$.
If it is a $T_1$ space, then it is necessarily a $\kappa$-exclusive space with $\kappa < \omega_0$.
This conclusion may help as a general guideline in constructing counterexamples.
We establish a more general result in the following theorem.

\begin{theorem}
	Let $(X, \tau)$ be a $\kappa$-exclusive space of Lindel\"of degree $L(X) = \kappa$, then $X$ is a D-space.
\end{theorem}

\begin{proof}
	The proof is similar to the proof of the previous theorem but with each $D$ having cardinality at most $\kappa$.
\end{proof}

Moreover, in a $\kappa$-exclusive space $(X, \tau)$ of Lindel\"of degree $L(X) = \kappa$, any closed discrete subset of $X$ has cardinality at most $\kappa$ and the proof of this fact is similar to the proof of proposition~\ref{closed-discrete-in-lindelof-is-countable}.

\begin{lemma}\label{extend-closed-discrete}
	Let $(X, \tau)$ be a $T_1$ topological space.
	If $F$ is a closed discrete subset of $X$ and $x$ is any point in $X \setminus F$, then $F \cup \{x\}$ is also closed and discrete.
\end{lemma}

\begin{proof}
	Assume that $F$ is a closed discrete subset of $X$, and let $x_0$ be any point in ${ X \setminus F }$.
	Since $\{x_0\}$ is closed, then the union $F \cup \{x_0\}$ is indeed closed too.
	To prove the discreteness of $F \cup \{x_0\}$, consider any open neighborhood $U_{x_0}$ of $x_0$ and observe that $U_{x_0} \setminus F$ is also an open neighborhood of $x_0$.
	Moreover, $U_{x_0} \cap (F \cup \{x_0\}) = \{x_0\}$.
	Now take any point $x \in F$, then since $F$ is discrete, we can find an open neighborhood $U_{x}$ of $x$ such that $U_{x} \cap F = \{x\}$.
	Thus, $U_{x}^{*} = U_{x} \setminus \{x_0\}$ is an open neighborhood of $x$ such that $U_{x}^{*} \cap (F \cup \{x_0\}) = \{x\}$ which ends the proof.
\end{proof}

Now, we will introduce fair spaces and then show that all fair spaces are D-space.
We will prove this by applying transfinite recursion to construct, for each open neighborhood assignment, a closed discrete kernel $D$.

\begin{definition}
	A topological space $(X, \tau)$ is said to be fair if it has the following property.
	For any collection $\mathcal{D}$ of nested closed discrete subsets of $X$, if $D = \bigcup \mathcal{D}$ is discrete, then $D$ is closed.
\end{definition}

\begin{theorem}\label{fair-space-is-d}
	Any $T_1$ fair space $(X, \tau)$ is a D-space.
\end{theorem}

\begin{proof}
	Let $N : X \rightarrow \tau$ be an open neighborhood assignment.
	At each ordinal $\xi < \gamma$ for some ordinal $\gamma$, we will construct, by transfinite recursion, a set $D_\xi$, a neighborhood refinement $N_\xi$ of $N$, and if possible (if not, then set $x_\xi = \phi$), a point $x_\xi$ such that:
	\begin{enumerate}
		\item\label{itm:fair-space-is-d-1} The set $D_{\xi} = \{ x_{\alpha} : \alpha \leq \xi \} \setminus \{ \phi \}$ is a closed discrete subset of $X$ such that $D_\alpha \subsetneq D_\beta$ for all $\alpha$ and $\beta$ where $\alpha < \beta \leq \xi$.
		\item\label{itm:fair-space-is-d-2} For all $d_0 \in D$, the only open set among $\{N_{\xi}(d) : d \in D_\xi\}$ which contains $d_0$ is $N_{\xi}(d_0)$.
		\item\label{itm:fair-space-is-d-3} $N_{\alpha}(d) = N_{\beta}(d)$ for all $d \in D_{\alpha}$ where $\alpha < \beta \leq \xi$.
	\end{enumerate}

	For the initial step of the recursion, take any $x_0$ in $X$ and let $N_0 = N$, then together with the set $D_0 = \{ x_0 \}$ all three of the conditions are satisfied.

	For the successor step, assume that the above three conditions hold for an ordinal $\beta$.
	If $X \setminus \bigcup\limits_{\alpha \leq \beta} N_{\beta}(x_\alpha) = \phi$, then the recursion ends.
	If not, take any $x_{\beta^{+}}$ in $X \setminus \bigcup\limits_{\alpha \leq \beta} N_{\beta}(x_\alpha)$ and define $D_{\beta^{+}}$ as $D_{\beta} \cup \{ x_{\beta^{+}} \}$.
	Since $D_{\beta}$ is a closed discrete set, then from lemma~\ref{extend-closed-discrete}, $D_{\beta^{+}}$ is also closed and discrete.
	Therefore,~\ref{itm:fair-space-is-d-1} holds.
	Since $x_{\beta^{+}} \in X \setminus \bigcup\limits_{\alpha \leq \beta} N_{\beta}(x_\alpha)$, then $x_{\beta^{+}} \notin N_{\beta}(d)$ for all $d \in D_{\beta}$.
	Moreover, the set $N_{\beta}(x_{\beta^{+}}) \setminus D_{\beta}$ is open since $D_{\beta}$ is closed.
	Define the neighborhood refinement $N_{\beta^{+}}$ of $N$ as follows.
	Let $N_{\beta^{+}}(x) = N_{\beta}(x)$ for all $x$ in $X \setminus \{ x_{\beta^{+}} \}$ and $N_{ \beta^{+} }(x_{\beta^{+}}) = N_{\beta}(x_{\beta^{+}}) \setminus D_{\beta}$.
	Hence,~\ref{itm:fair-space-is-d-2} follows for $N_{\beta^{+}}$.
	Condition~\ref{itm:fair-space-is-d-3} is a direct consequence of the definition of $N_{\beta^{+}}$.

	Finally, for the limit step, let $\lambda$ be any limit ordinal and assume that the above three conditions holds for all $D_\xi$, $N_\xi$, and $x_\xi$ where $\xi < \lambda$.
	First, let $D_{\lambda}^{*} = \bigcup\limits_{\xi < \lambda} D_\xi$ then define $N_{\lambda}^{*}$ as follows.
	Let $N_{\lambda}^{*}(x_\alpha) = N_{\alpha}(x_\alpha)$ for all $x_\alpha \in D_{\lambda}^{*}$ and let $N_{\lambda}^{*}(x) = N(x)$ for all ${ x \notin D_{\lambda}^{*} }$.
	We want to show that for each $x_\alpha$ in $D_{\lambda}^{*}$, the only open set among ${ \{ N_{\lambda}^{*}(d) : d \in D_{\lambda}^{*}\} }$ containing $x_\alpha$ is $N_{\lambda}^{*}(x_\alpha) = N_{\alpha}(x_\alpha)$.
	To prove this, assume not, i.e.\ that for some $\alpha < \lambda$, $N_{\lambda}^{*}(x_\alpha)$ contains another point other than $x_\alpha$ in $D_{\lambda}^{*}$, say $x_\beta$ for some ${ \beta \neq \alpha }$.
	If $\beta < \alpha$, then $N_{\lambda}^{*}(x_\beta) = N_{\beta}(x_\beta) = N_{\alpha}(x_\beta)$ (from~\ref{itm:fair-space-is-d-3}) and since the only open set among ${ \{ N_{\alpha}(d) : d \in D_\alpha \} }$ which contains $x_\beta$ is $N_{\alpha}(x_\beta)$ (from~\ref{itm:fair-space-is-d-2}), then $x_\beta$ must equal $x_\alpha$ or else $N_{\alpha}(x_\alpha)$ is another open set in $\{ N_{\alpha}(d) : d \in D_\alpha \}$ which contains $x_\beta$, a contradiction to the assumption that $x_\beta \neq x_\alpha$.
	On the other hand, if $\alpha < \beta$, then ${ N_{\lambda}^{*}(x_\alpha) = N_{\alpha}(x_\alpha) = N_{\beta}(x_\alpha) }$ (from~\ref{itm:fair-space-is-d-3}) and since the only open set in ${ \{ N_{\beta}(d) : d \in D_\beta \} }$ which contains $x_\beta$ is $N_{\beta}(x_\beta)$ (from~\ref{itm:fair-space-is-d-2}) then $x_\beta$ must equal $x_\alpha$ or else $N_{\beta}(x_\alpha)$ is another open set in $\{ N_{\beta}(d) : d \in D_\beta \}$ which contains $x_\beta$, a contradiction to the assumption that $x_\beta \neq x_\alpha$.
	Therefore,~\ref{itm:fair-space-is-d-2} holds for $N_{\lambda}^{*}$.
	To show that $D_{\lambda}^{*}$ is discrete, take any $x_\alpha$ in $D_{\lambda}^{*}$, then $N_{\lambda}^{*}(x_\alpha)$ is the only open set among $\{ N_{\lambda}^{*}(d) : d \in D_{\lambda}^{*} \}$ which contains $x_\alpha$.
	Hence, $N_{\lambda}^{*}(x_\alpha) \cap D_{\lambda}^{*} = \{ x_\alpha \}$; and thus, $D_{\lambda}^{*}$ is discrete.
	Since $D_{\lambda}^{*}$ is discrete and is the union of a collection of nested sets each of which are closed and discrete, then from the fact that $X$ is fair, $D_{\lambda}^{*}$ is also closed which implies that~\ref{itm:fair-space-is-d-1} holds for $D_{*}$.
	Furthermore,~\ref{itm:fair-space-is-d-3} holds as well from the way we defined $N_{\lambda}^{*}$.
	If $X \setminus \bigcup\limits_{\alpha < \lambda} N_{\lambda}^{*}(x_\alpha) = \phi$, then the recursion ends so set $D_{\lambda} = D_{\lambda}^{*}$, $N_{\lambda} = N_{\lambda}^{*}$, and $x_\lambda = \phi$.
	If not, take any $x_\lambda$ in $X \setminus \bigcup\limits_{\alpha < \lambda} N_{\lambda}^{*}(x_\alpha)$;
	hence, by a similar argument as to the successor step, one can obtain a neighborhood refinement $N_\lambda$ from $N_{\lambda}^{*}$ such that $D_\lambda = D_{\lambda}^{*} \bigcup \{ x_\lambda \}$ together with $N_\lambda$ satisfy all the three conditions.

	This recursion must end at some ordinal $\gamma$, since $D_\alpha \subseteq X$ and each $D_\alpha$ is different for all $\alpha$.
	If it didn't end, then there exists an order isomorphism $\mathcal{C} : \text{\textbf{Ord}} \rightarrow \mathcal{D}$ where $\text{\textbf{Ord}}$ is the class of all ordinals and $\mathcal{D}$ is $\{ D_\alpha: \alpha \in \text{\textbf{Ord}} \}$ ordered by set inclusion $\subseteq$.
	This implies that $\mathcal{D}$ is a proper class.
	However, $\mathcal{D} \subseteq \mathcal{P}(X)$, and hence $\mathcal{P}(X)$ is a proper class too which contradicts the power set axiom.

	Moreover, $X \setminus \bigcup N_{\gamma}(D_\gamma) = \phi$.
	Therefore, $D_{\gamma}$ is a kernel of $N_{\gamma}$, and thus by lemma~\ref{refinement-in-d-space}, $X$ is a D-space which ends the proof.
\end{proof}

Another general guideline here for constructing counterexamples is that if some space is not a D-space, then it is not a fair space.
Note that the real line with the usual topology $(\mathbb{R}, \tau_{u})$ is not a fair space, that is since the collection $\mathcal{C} = \{ F_i : i \in \mathbb{N} \}$ where $F_i = \{ \frac{1}{j} : j \leq i \}$ consists of nested closed discrete sets and the union $\bigcup \mathcal{C}$ is discrete;
however, it is not closed.

\section{On paracompact and metacompact D-spaces}\label{sec:on-paracompact-and-metacompact-d-spaces}

The following theorem shows us how we can relate the property D and paracompactness together using companion maps.

\begin{theorem}\label{paracompact-d-space-sufficient}
	Let $(X, \tau)$  be a $T_1$ topological space.
If for any continuous map ${ f : (X, \tau) \rightarrow (\mathcal{P}(X), \tau_{\mathcal{S}(X)}) }$ such that $f(x) \in \mathcal{U}(x)$ for all $x$ in $X$, there exists a continuous neighborhood refinement map $f_{*}$ of $f$ and a set $D \subseteq X$ for which the continuous map $g:(X, \tau) \rightarrow (\mathcal{P}(D), \tau_{\mathcal{S}(D)})$ defined as $g(x) = f_{*}(x) \cap D$ satisfies
	\begin{enumerate}
		\item $g^{-1}(\{ \phi \}) = \phi$ \label{paracompact-d-space-sufficient-1}
		\item $\lvert \bigcup g(N^{*}(x)) \rvert < \omega_0$ for all $x$ in $X$ \label{paracompact-d-space-sufficient-2}
	\end{enumerate}
	where $N^{*}$ is the neighborhood assignment associated with the companion map $f_{*}$, then $X$ is a paracompact D-space.
\end{theorem}

\begin{proof}
	Assume the hypothesis of the theorem.
	Let $f : (X, \tau) \rightarrow (\mathcal{P}(X), \tau_{\mathcal{S}(X)})$ be a continuous map such that $f(x) \in \mathcal{U}(x)$ for all $x$ in $X$, then there exists a continuous neighborhood refinement map $f_*$ of $f$ and a set $D \subseteq X$ for which the continuous map ${ g:(X, \tau) \rightarrow (\mathcal{P}(D), \tau_{\mathcal{S}(D)}) }$ defined as $g(x) = f_{*}(x) \cap D$ satisfies both~\ref{paracompact-d-space-sufficient-1} and~\ref{paracompact-d-space-sufficient-2}.
	Since $g(d) \subseteq \bigcup g(N^{*}(d))$, then $\lvert g(d) \rvert \leq \lvert \bigcup g(N^{*}(d)) \rvert < \omega_0$, i.e.\ $g(d) = \{ a \in D : d \in N^{*}(a) \}$ is finite for all $d$ in $D$ which implies that $g(d)$ is closed since $X$ is $T_1$.
	Define the neighborhood refinement $N^{**}$ of $N^{*}$ as $N^{**}(d) = N^{*}(d) \setminus (g(d) \setminus \{d\})$ for all $d$ in $D$ and $N^{**}(x) = N^{*}(x)$ for all $x$ in $X \setminus D$.
	Observe that the companion map $f_{**}$ of $N^{**}$, the set $D$, and the map $g^{*}:(X, \tau) \rightarrow (\mathcal{P}(D), \tau_{\mathcal{S}(D)})$ defined as $g^{*}(x) = f_{**}(x) \cap D$ all satisfy the statement of corollary~\ref{main-theorem-d-space-corollary}.
	Thus, $X$ is a D-space.
	To show that $X$ is paracompact, let $\mathcal{O} = \{ U_a : a \in A \}$ be any open cover of $X$.
	Since for each $x$ in $X$, there is an $a_x \in A$ such that $x \in U_{a_x}$, then define the assignment $N : X \rightarrow \tau$ as $N(x) = U_{a_x}$.
	Clearly, $N$ is an open neighborhood assignment, and hence has a companion map $f$ that satisfies $f(x) \in \mathcal{U}(x)$ for all $x \in X$.
	Therefore, there is a continuous neighborhood refinement map $f_{*}$ of $f$ and a set $D \subseteq X$ such that the continuous map $g$ defined in the hypothesis satisfies both~\ref{paracompact-d-space-sufficient-1} and~\ref{paracompact-d-space-sufficient-2}.
	Hence, by corollary~\ref{subcover-map}, the set $\mathcal{O}_r = \{ N^{*}(d) : d \in D \}$ is an open refinement of $\mathcal{O}$ and the map $g$ is the companion map of $N^{*}\restriction_{D}$ which is the open covering assignment inducing the open cover $\mathcal{O}_r$.
	It is obvious that $N^{*}$ and $g$ both satisfy the necessary condition of lemma~\ref{locally-finite-companion-maps}, and as a consequence, $\mathcal{O}_r$ is locally finite.
\end{proof}

Similarly, the following theorem shows how we can relate the property D and metacompactness using companion maps.

\begin{theorem}\label{metacompact-d-space-sufficient}
	Let $(X, \tau)$  be a $T_1$ topological space.
	If for any continuous map ${ f : (X, \tau) \rightarrow (\mathcal{P}(X), \tau_{\mathcal{S}(X)}) }$ such that $f(x) \in \mathcal{U}(x)$ for all $x$ in $X$, there exists a continuous neighborhood refinement map $f_{*}$ of $f$ and a set $D \subseteq X$ for which the continuous map $g:(X, \tau) \rightarrow (\mathcal{P}(D), \tau_{\mathcal{S}(D)})$ defined as $g(x) = f_{*}(x) \cap D$ satisfies
	\begin{enumerate}
		\item $g^{-1}(\{ \phi \}) = \phi$ \label{metacompact-d-space-sufficient-1}
		\item $\lvert g(x) \rvert < \omega_0$ for all $x$ in $X$ \label{metacompact-d-space-sufficient-2}
	\end{enumerate}
	where $N^{*}$ is the neighborhood assignment associated with the companion map $f_{*}$, then $X$ is a metacompact D-space.
\end{theorem}

\begin{proof}
	Assume that the hypothesis of the theorem holds.
	Let $f : (X, \tau) \rightarrow (\mathcal{P}(X), \tau_{\mathcal{S}(X)})$ be a continuous map such that $f(x) \in \mathcal{U}(x)$ for all $x$ in $X$, then there exists a continuous neighborhood refinement map $f_*$ of $f$ and a set $D \subseteq X$ for which the continuous map $g:(X, \tau) \rightarrow (\mathcal{P}(D), \tau_{\mathcal{S}(D)})$ defined as $g(x) = f_{*}(x) \cap D$ satisfies both~\ref{metacompact-d-space-sufficient-1} and~\ref{metacompact-d-space-sufficient-2}.
	Since $g(d) = \{ a \in D : d \in N^{*}(a) \}$ is finite for all $d$ in $D$, then it is clear that $g(d)$ is closed.
	Define the neighborhood refinement $N^{**}$ of $N^{*}$ as $N^{**}(d) = N^{*}(d) \setminus (g(d) \setminus \{d\})$ for all $d$ in $D$ and $N^{**}(x) = N^{*}(x)$ for all $x$ in $X \setminus D$.
	Observe that the companion map $f_{**}$ of $N^{**}$, the set $D$, and the map $g^{*}:(X, \tau) \rightarrow (\mathcal{P}(D), \tau_{\mathcal{S}(D)})$ defined as $g^{*}(x) = f_{**}(x) \cap D$ all satisfy the statement of corollary~\ref{main-theorem-d-space-corollary}.
	Thus, $X$ is a D-space.

	To show that $X$ is metacompact, let $C : A \rightarrow \tau$ be any open covering assignment, then $C(A) = \{ U_a : a \in A \}$ is an open cover.
	Since for each $x$ in $X$ there is an $a_x \in A$ such that $x \in U_{a_x}$, then define the assignment $N : X \rightarrow \tau$ as $N(x) = U_{a_x}$.
	Clearly, $N$ is an open neighborhood assignment, and hence has a companion map $f$ that satisfies $f(x) \in \mathcal{U}(x)$ for all $x \in X$.
	Therefore, there is a continuous neighborhood refinement map $f_{*}$ of $f$ and a set $D \subseteq X$ such that the continuous map $g$ defined in the hypothesis satisfies both~\ref{metacompact-d-space-sufficient-1} and~\ref{metacompact-d-space-sufficient-2}.
	Hence, by corollary~\ref{subcover-map}, the set $\mathcal{O}_r = \{ N^{*}(d) : d \in D \}$ is an open refinement of $C(A)$ and the map $g$ is the companion map of $N^{*}\restriction_{D}$ which is the open covering assignment.
	Clearly, for all $d \in D$, $N^{*}\restriction_{D}(d) \subseteq N(d) \subseteq U_{a_d} = C(d)$.
	Finally, as a direct application of theorem~\ref{metacompactness-by-companion-maps}, $X$ is metacompact.
\end{proof}

\clearpage
\chapter{Final Thoughts}\label{ch:final-thoughts}
\section{Comments}\label{sec:comments}

The central idea of this thesis was to find a new way to talk about covering properties that may bring new insights to the problems related to D-spaces.
Being inspired by question~\ref{arhangelskii-question} posed by Arhangel'skii, studying continuous maps of D-spaces could eventually lead to the solution of the problem whether every (regular) Lindel\"of space is a D-space.

As an attempt, we tried to obtain an equivalent definition of the property D that is more oriented towards mappings.
To accomplish this, we thought of a way to hide set-related notions;
e.g.\ ``closed discrete'', and making such properties follow from the properties of mappings instead.

The role of the puf topology was to encode the information about the indices of all the open sets that contained any given point in a continuous mapping.
This was the companion map which we've studied in depth.
Category theory was a language we used to write some structures in, as it was very tempting to use.
However, it was used as a conceptual tool and not as a problem-solving one.

After we found a way to define D-spaces using companion maps, we tried to test the effectiveness of the approach by applying it to cases where the property D was connected to other covering properties such as the Lindel\"of property, paracompactness, and metacompactness.
We were able to prove some statements;
however, we wish that we could arrive at stronger results.
We also proved other results which were independent of our approach to give some insights on how one should go about constructing counterexamples.

It is expected that natural questions could arise about the puf topology;
for example, what topological properties does it have?
Obviously, it doesn't have a good separation property since it doesn't have many open sets.
Moreover, its properties don't vary much from one set to another, especially if the sets are infinite, since the puf topology is related only to the set $X$ on which its induced and not to any topology defined on $X$.

Although the properties of the puf topology are not that interesting, it was important because of the role it played in defining the companion maps and since it could be constructed from any given set with no constraints.

\section{Future Research}\label{sec:future-research}

We ask the following questions that could motivate future research in the same path we took.

\begin{question}\label{better-than-puf}
	Could we introduce a topology that has better properties than the puf topology which may or may not depend on the topology on $X$ and that refines our approach?
\end{question}

We discussed the role and limitations of the puf topology in the previous section.
We would like to see whether there are ``nicer'' topological spaces that could render our approach using companion maps more effective.
For instance, we could look for a topological space that could be used to encode more information about the original topological space defined on $X$ such as its separation property.
This also brings us to the following question.

\begin{question}\label{better-than-puf-2}
	Instead of looking for a new topology, could we start from the puf topology and then apply a sequence of refinements to the base generated by the subbase $\mathcal{S}(X)$ to obtain a better topology?
\end{question}

Question~\ref{better-than-puf-2} could be more reasonable to ask, since there is a point to start searching from, then depending on the properties of the given topology on $X$, we can apply adjustments to the puf topology on $X$.

\begin{question}\label{more-constructs}
	What other results could we obtain by applying this approach to different constructs on D-spaces (or more generally on spaces with a covering property)?
\end{question}

We would like to see what our approach could provide when talking about different constructs on D-spaces such as unions and products of D-spaces.
Especially after our approach is refined with a better topology than the puf topology.

\begin{question}
	What other interesting constructions in \textbf{Top} could we obtain from this approach (with or without refining the puf topology)?
\end{question}

There is still more to look at concerning constructions and universal mapping properties in the category \textbf{Top} which are related to our approach.
We would like to see the extent to which studying such constructions could be helpful.

\begin{question}
	Could our approach make some proofs of existing results on D-spaces shorter and clearer?
\end{question}

It would be great to investigate how our approach could be useful in refining existing proofs and approaches to problems on D-spaces and covering properties in general.

Finally, what we wish is to solve the open problems we discussed either in the current setting or in a different one.
It could be that the current axioms of set theory are not sufficient to solve these problems.
Also, if we were unfortunate and due to Kurt G\"odel's incompleteness theorems in mathematical logic (see Smullyan, 1992)$\,^{\text{\cite{smullyan}}}$, statements like ``every regular Lindel\"of space is a D-space'' and ``every regular paracompact space is a D-space'' could be undecidable statements, i.e.\ statements being neither provable nor refutable in a specified deductive system unless they are taken as axioms which, in this case, become trivial.
Nonetheless, it could still be beneficial to develop this theory and try to apply it in different contexts.

\cleardoublepage

\appendix
\chapter*{Appendices}\label{ch:appendices}
\addcontentsline{toc}{chapter}{Appendices}
\renewcommand{\thesection}{A.\arabic{section}}
\section{Category theory and the category \textbf{Top}}\label{sec:category-theory-and-the-categorytextbf}

This is a concise introduction to category theory, and in particular to the category \textbf{Top} which we have worked with in this thesis.
According to Awodey (2010)$\,^{\text{\cite{awodey}}}$, one can think of category theory as ``the mathematical study of (abstract) algebras of functions''.
Another way to look at category theory, according to Leinster (2014)$\,^{\text{\cite{awodey}}}$, is that ``Category theory takes a bird's eye view of mathematics.
From high in the sky, details become invisible, but we can spot patterns that were impossible to detect from ground level''.
Yet another way to view category theory is that, instead of looking at a particular ``object'' and its internal properties such as a topological space, we study the relations that hold (external to the object) between that object and all the other objects in some universe, e.g.\ all topological spaces.
Such ``universal'' relations are called \textit{universal properties} and they are central to category theory.
To accomplish this, category theory heavily relies on a generalization of the concept of a structure-preserving mapping/transformation.

The history of category theory could be traced back to the \textit{Felix Klein's Erlanger Program} (Marquis, 2009)$\,^{\text{\cite{marquis}}}$.
The fundamental idea in Klein's program was to characterize geometries by their groups of transformations.
Thus, a geometrical property is a property that is invariant under these transformations.
In other words, these transformations encoded properties of the geometry.
Therefore, two geometries are the same if there is some bijection between the underlying spaces and an isomorphism between their transformation groups.

As a continuation of the Erlanger Program, Samuel Eilenberg and Saunders Mac Lane sought to develop a language to characterize different types of mathematical structures by their ``admissible transformations''.
The first work on category theory was the ``General theory of natural equivalences'' (Mac Lane S. and Eilenberg S., 1945)$\,^{\text{\cite{maclane-eilenberg}}}$.
They saw their work as a generalization and an analogy to Klein's program.
The motivation behind their work was to formalize the intuitive notion of ``natural equivalences'', i.e.\ equivalences that didn't depend on arbitrary choices.
For example, there is a natural equivalence between a vector space $V$ and its double dual $V^{**}$ in the sense that the equivalence doesn't depend on the arbitrary choice of the basis.
Since equivalence between vector spaces is demonstrated through linear maps (which is the structure-preserving map in the class of all vector spaces), they needed to introduce a generalization to structure-preserving maps.
They started by defining the concept of a category (an aggregate of objects and mappings between them) which we will define, the concept of a \textit{functor} which is a structure-preserving map between categories, and the concept of a \textit{natural transformation} which is a transformation from one functor to another while respecting the internal structure of the categories involved.
Hence, we can demonstrate the natural equivalence between a vector space and its double dual by showing that there is a natural isomorphism between the identity functor on the category of vector spaces and the functor that takes each vector to its double dual.

Category theory is applied to several mathematical disciplines, such as algebraic geometry, algebraic topology, homological algebra, and even to logic.
Some see Category theory as a unifying language for mathematics that is more natural than set theory because of its conceptual characteristics.
As to the foundations of category theory, some base category theory on a set theory such as ZFC or NBG using the notion of ``conglomerates'', and other axiomatize category theory without set theory such as in ETCC (the elementary theory of the category of categories).
For more on the foundations of category theory the reader can consult (Ad\'amek et al. 2004)$\,^{\text{\cite{adamek-herrlich-strecker}}}$.

There are several definitions of the notion of a category;
however, we will use a common definition as in Awodey (2010)$\,^{\text{\cite{awodey}}}$.


\begin{definition}
	A category consists of the following
	\begin{enumerate}
		\item Objects: $A$, $B$, \dots
		\item Arrows: $f$, $g$, \dots
		\item For each arrow $f$, there are given objects $dom(f)$ and $cod(f)$ called the domain and codomain of $f$.
		We write $f: A \rightarrow B$ to indicate that $A = dom(f)$ and $B = cod(f)$.
		\item For each arrows $f: A \rightarrow B$, $g: B \rightarrow C$, i.e.\ with $cod(f) = dom(g)$, there is an arrow $g \circ f: A \rightarrow C$ called the composite of $f$ and $g$.
		\item For each object $A$, there is an arrow $1_A : A \rightarrow A$ called the identity arrow of $A$.
	\end{enumerate}
	Such that the above satisfy
	\begin{enumerate}
		\item Associativity: $h \circ (g \circ f) = (h \circ g) \circ f$ for all arrows $f : A \rightarrow B$, $g : B \rightarrow C$, and $h : C \rightarrow D$.
		\item Unit: $f \circ 1_A  = f = 1_B \circ f$ for all arrows $f : A \rightarrow B$.
	\end{enumerate}
\end{definition}

As an important note here, arrows which are also called morphisms need not be functions.
Although arrows may be functions, there are cases where a category has arrows as something other than functions such as set relations or something else completely unrelated to sets which satisfies the properties above.
In other words, arrows/morphisms are generalizations of set functions.
Furthermore, the objects themselves need not be sets.
This makes category theory a strong tool for making abstractions.

An example of a category is the category \textbf{Top} where objects are topological spaces and arrows are continuous mappings between these spaces.
Categories with objects as sets with some structure and morphisms as structure-preserving mappings such as \textbf{Top} are usually called categories of structured sets.
More generally, categories with objects as sets are usually called concrete categories.

Another important note is that arrows only matter in category theory.
There is an equivalent definition for categories that only consists of arrows, i.e.\ without objects;
however, it is much easier to deal with the definition above.

To introduce what is meant by a commutative diagram, we show an example of a diagram below.
\begin{center}
	\begin{tikzcd}
		A \arrow[rd, "z"'] \arrow[r, "f"] & B \arrow[d, "g"] \\
												 & C
	\end{tikzcd}
\end{center}

If it is not necessary that $z = g \circ f$, then the diagram doesn't commute.
Otherwise, if $z = g \circ f$, then the diagram is commutative, and in particular becomes a commutative triangle.
A commutative diagram is a diagram for which any two different sequences of compositions from one node to another are equal.
In our case, we have $z$ alone, and then the composition of $g$ and $f$, i.e.\ $g \circ f$ as two different sequences of compositions from $A$ to $C$.
These diagrams makes it easier to understand constructions in category theory.
Also, a common technique of proofs in category theory is called diagram chasing in which one proves that two arrows are equal using the fact that the diagram is commutative and obtaining different sequences of compositions from a node to another.

Similar to continuous maps of topological spaces, linear maps of vector spaces, and homomorphisms of groups, there is also a ``structure-preserving map'' between categories which is the functor.

\begin{definition}
	For any two categories $\mathcal{C}$ and $\mathcal{D}$, a functor from the category $\mathcal{C}$ to the category $\mathcal{D}$ is a mapping of objects to objects and arrows to arrows such that:
	\begin{enumerate}
		\item $F(f_\mathcal{C} : A_{\mathcal{C}} \rightarrow B_{\mathcal{C}}) = F(f_\mathcal{C}) : F(A_{\mathcal{C}}) \rightarrow F(B_\mathcal{C}) = f_\mathcal{D} : A_{\mathcal{D}} \rightarrow B_{\mathcal{D}}$
		\item $F(1_{A_{\mathcal{C}}}) = 1_{F(A_{\mathcal{C}})} = 1_{A_{\mathcal{D}}}$
		\item $F(g_{\mathcal{C}} \circ f_{\mathcal{C}}) = F(g_{\mathcal{C}}) \circ F(f_{\mathcal{C}}) = g_{\mathcal{D}} \circ f_{\mathcal{D}}$
	\end{enumerate}
\end{definition}


The commutative diagrams below illustrates the definition above.

\begin{center}
	\begin{tikzcd}[transform shape, nodes={scale=1.1}]
		\mathcal{C} \arrow[rr, "F"] & & \mathcal{D}
	\end{tikzcd}

	\vspace{1cm}

	\begin{tikzcd}[transform shape, nodes={scale=1.1}]
		A_{\mathcal{C}} \arrow[loop left, "1_{A_{\mathcal{C}}}"] \arrow[r, "f_{\mathcal{C}}"] \arrow[rd, "g_{\mathcal{C}} \circ f_{\mathcal{C}}"'] & B_{\mathcal{C}} \arrow[loop right, "1_{B_{\mathcal{C}}}"] \arrow[d, "g_{\mathcal{C}}"] & & &  &
		A_{\mathcal{D}} \arrow[loop left, "1_{A_{\mathcal{D}}}"] \arrow[r, "f_{\mathcal{D}}"] \arrow[rd, "g_{\mathcal{D}} \circ f_{\mathcal{D}}"'] & B_{\mathcal{D}} \arrow[loop right, "1_{B_{\mathcal{D}}}"] \arrow[d, "g_{\mathcal{D}}"] \\
		& C_{\mathcal{C}} \arrow[loop below, "1_{C_{\mathcal{C}}}"] &  &  & & & C_{\mathcal{D}} \arrow[loop below, "1_{C_{\mathcal{D}}}"]
	\end{tikzcd}
\end{center}

Every category $\mathcal{C}$ has the identity functor $1_{\mathcal{C}}$ which takes each object to itself and each arrow to itself.
In general, any functor from a category to itself is called an endofunctor.
Another example of a functor is the forgetful functor $U$ from \textbf{Top} to the category \textbf{Set} (which consists of sets as objects and set functions as arrows) that takes each topological space $(X, \tau)$ to its underlying set $X$, forgetting its structure.

A category theoretical notion that captures equivalence in the usual sense such as bijections between sets and homeomorphisms between topological spaces is the isomorphism.

\begin{definition}
	In a category $\mathcal{C}$, an arrow $f : A \rightarrow B$ is called an isomorphism if there exists an arrow $g : B \rightarrow A$ such that $g \circ f = 1_A$ and $f \circ g = 1_B$.
\end{definition}

Obviously, in the category \textbf{Top}, isomorphisms are just the homeomorphisms between topological spaces.

There are several constructions on categories.
One construction is the opposite (dual) category.
For any category $\mathcal{C}$, $\mathcal{C}^{\text{op}}$ is the dual category which consists of the same objects as $\mathcal{C}$ but with all the arrows reversed.
In other words, if $f : A \rightarrow B$ is an arrow in $\mathcal{C}$, then $f^{*} : B \rightarrow A$ is an arrow in $\mathcal{C}^{\text{op}}$.
Note that if $f : A \rightarrow B$ and $g : B \rightarrow C$ are any two arrows, then the dual of $g \circ f$ is $f^{*} \circ g^{*}$.
In category theory, if a statement is true in some category, then replacing each notion with its dual (e.g.\ reversing the arrows) in that statement results in a dual statement which is true in the opposite category.
For more on this topic, the reader should consult Awodey (2010)$\,^{\text{\cite{awodey}}}$ or Mac Lane, S. (1971)$\,^{\text{\cite{maclane}}}$.

Another common construction is the arrow category $\mathcal{C}^{\rightarrow}$ which has objects as arrows of $\mathcal{C}$ and arrows as commutative squares.
Thus, if $f:A \rightarrow B$ and $f':A' \rightarrow B'$ are any two arrows, then an arrow from $f$ to $f'$ is a pair of arrows $(g:A:\rightarrow A', g': B \rightarrow B')$ that makes the diagram below commute.

\begin{center}
	\begin{tikzcd}[transform shape, nodes={scale=1.2}]
		A \arrow[r, "f"] \arrow[d, "g"'] & B \arrow[d, "g'"] \\
		A' \arrow[r, "f'"']              & B'
	\end{tikzcd}
\end{center}

Note that the arrow $(g, g')$ which is a pair of arrows is not a function in the usual set theoretic sense.
The coslice category is also one of the common constructions which was already defined in section~\ref{sec:related-structures-in-the-categorytextbf} and is the dual notion of the slice category (see Awodey, 2010)$^{\text{\cite{awodey}}}$.

If we base category theory on some set theory, then the category is said to \textbf{small} if the collection of objects and the collection of arrows are sets.
Otherwise, it is said to be \textbf{large}.
A category $\mathcal{C}$ is said to be \textbf{locally small} if for all objects $X$ and $Y$ in $\mathcal{C}$, the collection $\text{Hom}_\mathcal{C}(X, Y) = \{ f\text{ is an arrow in } \mathcal{C} : dom(f) = X \wedge cod(f) = Y \}$ is a set.
As an example, the category $\textbf{Top}$ is locally small.
Moreover, $\text{Hom}_\mathcal{C}(X, Y)$ is called a ``Hom set''.

Now, we will discuss some category theoretic definitions which are characterizations of properties of objects and arrows in terms of other objects and arrows.
This method of characterization is usually regarded as ``external'' or structural as opposed to ``internal''.
A common way to do this is to introduce a universal mapping property.
For more motivation on these definitions and intuitive explanations, the reader should consult Awodey (2010)$\,^{\text{\cite{awodey}}}$.
In particular, understanding Awodey's ``no junk'' and ``no noise'' principles is very beneficial.

\begin{definition}
	In a category $\mathcal{C}$, an arrow $f : A \rightarrow B$ is said to be a monomorphism if for any two arrows $g : C \rightarrow A$ and $h : C \rightarrow A$ such that $f \circ g = f \circ h$, then $g = h$.
	\begin{center}
		\begin{tikzcd}
			C \arrow[r, "g", shift left=1ex] \arrow[r, "h"', shift right=1ex] & A \arrow[r, "f"] & B
		\end{tikzcd}
	\end{center}
\end{definition}

Monomorphisms in \textbf{Top} are injective continuous maps.

\begin{definition}
	In a category $\mathcal{C}$, an arrow $f : A \rightarrow B$ is said to be an epimorphism if for any two arrows $g : B \rightarrow C$ and $h : B \rightarrow C$ such that $g \circ f = h \circ f$, then $g = h$.
	\begin{center}
		\begin{tikzcd}
			A \arrow[r, "f"] & B \arrow[r, "g", shift left=1ex] \arrow[r, "h"', shift right=1ex] & C
		\end{tikzcd}
	\end{center}
\end{definition}

Epimorphisms in \textbf{Top} are surjective continuous maps.

\begin{definition}
	In a category $\mathcal{C}$, an object $0$ is said to be an initial object if for any object $C$, there is a unique morphism $0 \rightarrow C$.
	Moreover, an object $1$ is said to be terminal if for any object $C$, there is a unique morphism $C \rightarrow 1$.
\end{definition}

In \textbf{Top}, the initial object is the empty topological space, i.e.\ $(\phi, \{\phi\})$, and the terminal object is any topological space generated by a singleton, i.e.\ $(\{ a \}, \{\phi, \{ a \}\})$.

\begin{definition}
	In a category $\mathcal{C}$, the product of $A$ and $B$ consists of an object $P$ and two arrows $\pi_A : P \rightarrow A$ and $\pi_B : P \rightarrow B$, usually called projections, such that given any diagram of the form
	\begin{center}
		\begin{tikzcd}
			A & X \arrow[r, "p_B"] \arrow[l, "p_A"'] & B
		\end{tikzcd}
	\end{center}
	there is a unique map $u : X \rightarrow P$ that makes the following diagram commute.
	\begin{center}
		\begin{tikzcd}[transform shape, nodes={scale=1.2}]
			  & X \arrow[ld, "p_A"'] \arrow[rd, "p_B"] \arrow[d, "u", dashed] &   \\
			A & P \arrow[r, "\pi_B"] \arrow[l, "\pi_A"']                      & B
		\end{tikzcd}
	\end{center}
\end{definition}

In the category \textbf{Top}, the product of any two topological spaces $(X, \tau_X)$ and $(Y, \tau_Y)$ is the product topology on $X$ and $Y$.

\begin{definition}
	A category $\mathcal{C}$ is said to have all finite products, if it has a terminal object and all binary products (and thus all finite products).
\end{definition}

Indeed, the category \textbf{Top} has all finite products.


\begin{definition}
	In any category $\mathcal{C}$, given parallel arrows
	\begin{center}
		\begin{tikzcd}[transform shape, nodes={scale=1.2}]
			A \arrow[r, "f", shift left=1ex] \arrow[r, "g"', shift right=1ex] & B
		\end{tikzcd}
	\end{center}
	An equalizer of $f$ and $g$ consists of an object $E$ and an arrow $e : E \rightarrow A$, universal, such that $f \circ e = g \circ e$.
	That is, given any $z: Z \rightarrow A$ with $f \circ z = g \circ z$, there is a unique arrow $u : Z \rightarrow E$ such that the following diagram commutes.
	\begin{center}
		\begin{tikzcd}[transform shape, nodes={scale=1.2}]
			E \arrow[r, "e"]                          & A \arrow[r, "f", shift left=1ex] \arrow[r, "g"', shift right=1ex] & B \\
			Z \arrow[u, "u", dashed] \arrow[ru, "z"'] &   &
		\end{tikzcd}
	\end{center}
\end{definition}

One can look at the previous definition in the following manner.
Regarding sets, we want to find the largest set $E$ that makes the two functions $f$ and $g$ equal when they are restricted to that set.
So, given any other smaller set $Z$ that makes the two functions equal, the set $Z$ will be definitely contained in $E$.
Thus, one can look at the arrow $u : Z \rightarrow E$ as implying that the information in $Z$ is somehow contained in $E$.
This is a very convenient way to understand universal mapping properties.

In the category \textbf{Top}, equalizers are subspaces which make the continuous functions equal in the set theoretic sense.
A nice example of an equalizer is subspace inclusion.
Let $(X, \tau)$ be a topological space and $A \subseteq X$ be any subspace.
Let $(Y, \tau_{\text{ind}})$ be the topological space consisting of $Y = \{ 0, 1 \}$ with the indiscrete topology.
Define $f : X \rightarrow Y$ as $f(x) = 1$ for all $x$ in $X$.
Moreover, define $g : X \rightarrow Y$ as $g(x) = 1$ for all $x \in A$ and $g(x) = 0$ for all $x \in X \setminus A$.
The equalizer of $f$ and $g$ is the subspace $A$ together with the inclusion map $i : A \rightarrow X$, i.e.\ $i(a) = a$ for all $a$ in $A$.

Another way to characterize subspaces is by the following universal mapping property (see Manetti, 2015)$\,^{\text{\cite{manetti}}}$, as stated in the following theorem.


\begin{theorem}
	Let $(X, \tau_X)$ be a topological space, $A \subseteq X$ be any set, and ${ i: A \rightarrow X }$ be the inclusion map.
	The topological space on $A$ which is the subspace topology is characterized by the following universal mapping property.
	For every topological space $(Y, \tau_Y)$ and every function $f : Y \rightarrow A$, $f$ is continuous iff $i \circ f$ is continuous.
	\begin{center}
		\begin{tikzcd}[transform shape, nodes={scale=1.2}]
            			                             & X                 \\
			Y \arrow[r, "f"'] \arrow[ru, "i\circ f"] & A \arrow[u, "i"']
		\end{tikzcd}
	\end{center}
\end{theorem}

Finally, we introduce the definition of pullbacks.

\begin{definition}
	In any category $\mathcal{C}$, given any arrows $f:A \rightarrow C$ and $g: B \rightarrow C$, the pullback of $f$ and $g$ consists of an object $P$ and arrows $p_A : P \rightarrow A$, $p_B : P \rightarrow B$ such that $f \circ p_A = g \circ p_B$ and universal with this property.
	That is, given any object $Z$ and arrows $z_A : Z \rightarrow A$, $z_B : Z \rightarrow B$ with $f \circ z_A = g \circ z_B$, there exists a unique arrow $u : Z \rightarrow P$ with $z_A = p_A \circ u$ and $z_B = p_B \circ u$ as illustrated in the following commutative diagram.
	\begin{center}
		\begin{tikzcd}
			Z \arrow[rdd, "z_A"', bend right] \arrow[rrd, "z_B", bend left] \arrow[rd, "u" description, dashed] &                                      &                  \\
																												& P \arrow[r, "p_B"] \arrow[d, "p_A"'] & B \arrow[d, "g"] \\
																												& A  \arrow[r, "f"']                   & C
		\end{tikzcd}
	\end{center}
\end{definition}

According to Awodey (2010)$\,^{\text{\cite{awodey}}}$, the pullback could be equivalently defined as follows.
We take the product $\Pi$ of both $A$ and $B$, then the equalizer $P$ of $f \circ \pi_A$ and $g \circ \pi_B$ where $\pi_A$ and $\pi_B$ are the projection maps of $\Pi$.
The equalizer $P$ is the pullback of $f$ and $g$.

The final point we will discuss regarding category theory is the notion of a natural transformation.
As we said before, a natural transformation is a mapping from one functor to another which respects the structures of the categories involved.
The formal definition would read as follows.


\begin{definition}
	For categories $\mathcal{C}$ and $\mathcal{D}$, let $F : \mathcal{C} \rightarrow \mathcal{D}$ and $G : \mathcal{C} \rightarrow \mathcal{D}$ be any two functors.
	A natural transformation $\upsilon : F \rightarrow G$ is a family of arrows in $\mathcal{D}$ of the following form
	\[
		\upsilon_C : F(C) \rightarrow G(C)
	\]
	,for each $C \in \mathcal{C}$, such that for any arrow $f : C \rightarrow C\,'$ in $\mathcal{C}$, it holds that
	\[
		\upsilon_{C\,'} \circ F(f) = G(f) \circ \upsilon_{C}
	\]
	, i.e\ the following diagram commutes in $\mathcal{D}$

	\begin{center}
		\begin{tikzcd}[transform shape, nodes={scale=1.2}]
			F(C) \arrow[r, "\upsilon_C"] \arrow[d, "F(f)"'] & G(C) \arrow[d, "G(f)"] \\
			F(C\,') \arrow[r, "\upsilon_{C\,'}"']               & G(C\,')
		\end{tikzcd}
	\end{center}
\end{definition}

For any two categories $\mathcal{C}$ and $\mathcal{D}$, the functor category $F(\mathcal{C}, \mathcal{D})$ has objects as functors $F : \mathcal{C} \rightarrow \mathcal{D}$ and arrows as natural transformations between these functors $\upsilon: F \rightarrow G$.
Definitely, there is an identity natural transformation $1_F : F \rightarrow F$ for each functor $F$.
Finally, a natural isomorphism is an isomorphism in the functor category.


\addcontentsline{toc}{chapter}{References}

\begin{thebibliography}{99}
	\bibitem{adamek-herrlich-strecker}
		Ad\'amek, J., Herrlich, H., and Strecker, G. E. (2004), \textbf{Abstract and Concrete Categories, The Joy of Cats}, Online Edition.
		Electronic Version.
	\bibitem{arhangelskii-2}
		Arhangel'skii, A. V. (2005), D-spaces and covering properties. \textbf{Topology and its Applications}, 146--147, 437--449.
	\bibitem{arhangelskii-1}
		Arhangel'skii, A. V. (2004), D-spaces and finite unions. \textbf{Proceedings of the American Mathematical Society}, 132, No. 7, 2163--2170.
	\bibitem{arhangelskii-buzyakova}
		Arhangel'skii, A. V. and Buzyakova R. Z. (2002), Addition theorems and D-spaces. \textbf{Commentationes Mathematicae Universitatis Carolinae}, 43, No. 4, 653--663.
	\bibitem{awodey}
		Awodey, S. (2010), \textbf{Category Theory}, Second Edition.
		Oxford University Press.
		Oxford Logic Guides, 52.
	\bibitem{borges}
		Borges, R. C. (1966), On stratifiable spaces. \textbf{Pacific Journal of Mathematics}, 17, No. 1, 1--16.
	\bibitem{borges-wehrly-correction}
		Borges, R. C. and Wehrly, C. A. (1998), Correction: ``Another study of D-spaces''. \textbf{Questions and Answers in General Topology}, 16, 77--78.
	\bibitem{borges-wehrly-qa}
		Borges, R. C. and Wehrly, C. A. (1996), Another study of D-spaces. \textbf{Questions and Answers in General Topology}, 14, 73--76.
	\bibitem{borges-wehrly}
		Borges, R. C. and Wehrly, C. A. (1991), A Study of D-spaces. \textbf{Topology Proceedings}, 16, 7--15.
	\bibitem{buzyakova-2}
		Buzyakova, R. Z. (2004), Hereditary D-property of function spaces over compacta. \textbf{Proceedings of the American Mathematical Society}, 132, No. 11, 3433--3439.
	\bibitem{buzyakova-1}
		Buzyakova, R. Z. (2002), On D-property of strong $\Sigma$ spaces. \textbf{Commentationes Mathematicae Universitatis Carolinae}, 43, No. 3, 493--495.
	\bibitem{engelking}
		Engelking, R. (1989), \textbf{General Topology}, Revised and Completed Edition.
		Heldermann Verlag Berlin.
		Sigma Series in Pure Mathematics, 6.
	\bibitem{fleissner-stanley}
		Fleissner, W. G. and Stanley, A. M. (2001), D-spaces. \textbf{Topology and its Applications}, 114, 261--271.
	\bibitem{gruenhage}
		Gruenhage, G. (2011), A survey of D-spaces. \textbf{Contemporary Mathematics}, 533, 13--28.
	\bibitem{heath-lutzer-zenor}
		Heath, R. W., Lutzer, D. J., and Zenor.
		P. L. (1973), Monotonically Normal Spaces. \textbf{Transactions of the American Mathematical Society}, 178, 481--493.
	\bibitem{henry}
		Henry, M. (1971), Stratifiable Spaces, Semi-stratifiable spaces, and their relation through mappings. \textbf{Pacific Journal of Mathematics}, 37, No. 3, 697--700.
	\bibitem{jech} 
		Jech, T. (2002), \textbf{Set Theory}, The Third Millennium Edition, revised and expanded.
		Springer-Verlag Berlin Heidelberg.
		Springer Monographs in Mathematics.
	\bibitem{leinster}
		Leinster, T. (2014), \textbf{Basic Category Theory}, First Edition.
		Cambridge University Press.
		Cambridge Studies in Advanced Mathematics, 143.
	\bibitem{maclane}
		Mac Lane, S. (1971), \textbf{Categories for the working mathematician}, Second Edition.
		Springer-Verlag New York.
		Graduate Texts in Mathematics, 5.
    \bibitem{maclane-eilenberg}
        Mac Lane, S. and Eilenberg, S. (1945), General Theory of Natural Equivalences. \textbf{Transactions of the American Mathematical Society}, 58, 231--294.
	\bibitem{manetti}
		Manetti, M. (2015), \textbf{Topology}, First Edition.
		Springer International Publishing.
		La Matematica per il 3+2, 91.
    \bibitem{marquis}
        Marquis, J.-P. (2009), \textbf{From a Geometrical Point of View, A Study of the History and Philosophy of Category Theory}, First Edition.
		Springer Netherlands.
		Logic, Epistemology, and the Unity of Science, 14.
	\bibitem{pearl}
		Pearl, E. (2007), \textbf{Open Problems in Topology II}, First Edition.
		Elsevier Science B.V.
	\bibitem{peng-2}
		Peng, L.-X. (2017), Some sufficiency conditions for D-spaces, dually discrete and its applications. \textbf{Topology and its Applications}, 217, 81--94.
	\bibitem{peng-1}
		Peng, L.-X. (2012), The D-property which relates to certain covering properties. \textbf{Topology and its Applications}, 159, 869--876.
	\bibitem{smullyan}
		Smullyan, R. M. (1992), \textbf{Godel's Incompleteness Theorems}, First Edition.
		Oxford University Press.
		Oxford Logic Guides, 19.
	\bibitem{soukup-szeptycki}
		Soukup, D. T. and Szeptycki, P. J. (2012), A counterexample in the theory of D-spaces. \textbf{Topology and its Applications}, 159, 2669--2678.
	\bibitem{van-douwen-lutzer}
		Van Douwen, E. K. and Lutzer, D. J. (1997), A Note on Paracompactness in Generalized Ordered Spaces. \textbf{Proceedings of the American Mathematical Society}, 125, No. 4, 1237--1245.
	\bibitem{van-douwen-pfeffer}
		Van Douwen, E. K. and Pfeffer, W. V. F. (1979), Some properties of the Sorgenfrey line and related spaces. \textbf{Pacific Journal of Mathematics}, 81, 371--377.
	\bibitem{zhang-shi}
		Zhang, H. and Shi, W.-X. (2012), A note on D-spaces. \textbf{Topology and its Applications}, 159, 248--252.
\end{thebibliography}
\end{document}